%% file: subdivision.tex
\documentclass{timgad}
\pdfoutput=1
\usepackage{ifpdf} 

\title{On inverse powers of graphs and topological implications of Hedetniemi's conjecture\texorpdfstring{\footnotetext{This work has been supported by the National Science Centre, Poland, via the PRELUDIUM grant 2016/21/N/ST6/00475. The author is also supported by the Foundation for Polish Science via the START stipend programme.}}{}}
\date{December 2017, updated March 2019}
\author{Marcin Wrochna}
\affil{Institute of Informatics, University of Warsaw, Poland\\
	\textnormal{\texttt{m.wrochna@mimuw.edu.pl}}}

\usepackage[backend=bibtex,bibencoding=ascii,style=alphabetic,bibstyle=alphabetic,giveninits=true,doi=false,isbn=false,url=false,maxbibnames=99]{biblatex}
\renewbibmacro{in:}{}
\AtEveryBibitem{\clearname{editor}}

\newbibmacro{string+link}[1]{%
	\iffieldundef{ee}
	{\iffieldundef{url}
		{\iffieldundef{doi}
			{\iffieldundef{eprint}
				{#1}
				{\href{http://arxiv.org/abs/\thefield{eprint}}{#1}}}
			{\href{http://dx.doi.org/\thefield{doi}}{#1}}}
		{\href{\thefield{url}}{#1}}}
	{\href{\thefield{ee}}{#1}}} 
\DeclareFieldFormat{title}{\usebibmacro{string+link}{\mkbibemph{#1}}}
\DeclareFieldFormat[article,inproceedings,inbook,incollection,mastersthesis,thesis]{title}{\usebibmacro{string+link}{\mkbibquote{#1}}}

\renewcommand*{\bibfont}{\footnotesize}

\usepackage{tikz}
\usetikzlibrary{shapes,backgrounds,plothandlers,plotmarks,calc,arrows,fadings,patterns,decorations.pathmorphing,decorations.pathreplacing,fadings,decorations.markings,arrows.meta}

\newcommand{\NN}{\ensuremath{\mathbb{N}}}
\newcommand{\RR}{\ensuremath{\mathbb{R}}}
\newcommand{\ZZ}{\ensuremath{\mathbb{Z}}}
\newcommand{\Cc}{\ensuremath{\mathcal{C}}}
\DeclareMathOperator{\Hom}{Hom}
\DeclareMathOperator{\im}{im}
\DeclareMathOperator{\avg}{avg}
\DeclareMathOperator{\CN}{CN}
\DeclareMathOperator{\St}{St}
\DeclareMathOperator{\st}{st}
\DeclareMathOperator{\coind}{coind}
\DeclareMathOperator{\ind}{ind}

\newcommand\join{\mathbin{\begin{tikzpicture}[xscale=0.25,yscale=0.2,baseline=-1]
	\draw (0,0)--(1,0)--(0,1)--(1,1)--(0,0);
\end{tikzpicture}}}
\newcommand\notjoin{\mathbin{\begin{tikzpicture}[xscale=0.25,yscale=0.2,baseline=-1]
	\draw (0,0)--(1,0)--(0,1)--(1,1)--(0,0);
	\draw (0.2,-0.4)--(0.8,1.4);
\end{tikzpicture}}}
\newcommand{\dmatch}{\mathbin{\scalebox{0.7}{$\bigtriangleup$}}} 

\newcommand{\vL}[1]{{#1}{}\sp{\circ}}
\newcommand{\vR}[1]{{#1}{}\sp{\bullet}}
\newcommand{\vU}[1]{{#1}{}\sp{\circledast}}
\newcommand{\unk}{\ensuremath{\circledast}}
\newcommand{\vvL}[1]{{#1}_{\circ}}
\newcommand{\vvR}[1]{{#1}_{\bullet}}
\newcommand{\tup}[1]{\bar{#1}}

\newcommand{\Sphere}{\mathcal{S}}
\newcommand{\zeq}{\ensuremath{\simeq_{\ZZ_2}}} 

\newcommand{\SomeF}{\Pi}
\newcommand{\LeftF}{\Lambda}
\newcommand{\CentralF}{\Gamma}
\newcommand{\RightF}{\Omega}
\newcommand\PathLF[1]{\ensuremath{%
\LeftF_{{#1}}%
}}
\newcommand\PathCF[1]{\ensuremath{%
\CentralF_{{#1}}%
}}
\newcommand\PathRF[1]{\ensuremath{%
\RightF_{{#1}}%
}}
\newcommand\Bx[1]{\mathrm{Box}(#1)}
\newcommand\BBx[1]{|\mathrm{Box}(#1)|}

\begin{document}

\maketitle

\begin{abstract}
We consider a natural graph operation $\PathRF{k}$ that is a certain inverse (formally: the right adjoint) to taking the $k$-th power of a graph.
We show that it preserves the topology (the $\ZZ_2$-homotopy type) of the box complex, a basic tool in applications of topology in combinatorics.
Moreover, we prove that the box complex of a graph $G$ admits a $\ZZ_2$-map (an equivariant, continuous map) to the box complex of a graph $H$ if and only if the graph $\PathRF{k}(G)$ admits a homomorphism to $H$, for high enough $k$.

This allows to show that if Hedetniemi's conjecture on the chromatic number of graph products were true for $n$-colorings, then the following analogous conjecture in topology would also be true:
If $X,Y$ are $\ZZ_2$-spaces (finite $\ZZ_2$-simplicial complexes) such that $X \times Y$ admits a $\ZZ_2$-map to the $(n-2)$-dimensional sphere, then $X$ or $Y$ itself admits such a map.
We discuss this and other implications, arguing the importance of the topological conjecture. 
\end{abstract}

\section{Introduction}\label{sec:intro}
\input{intro.tex}

\section{Preliminaries}\label{sec:prelims}
\input{prelims}

\section{Multiplicativity of graphs and spaces}\label{sec:mult}
\input{mult.tex}

\section{Proof of Theorem~\ref{thm:multImplies} and Theorem~\ref{thm:characterization}}\label{sec:multproof}
\input{smallproofs.tex}

\section{Proof of the Equivalence Theorem~\ref{thm:equiv}}\label{sec:equiv}
\input{equiv.tex}

\section{Proof of the Approximation Theorem~\ref{thm:approximation}}\label{sec:approx}
\input{approx.tex}

\section*{Acknowledgments}
The author thanks Claude Tardif for many insightful conversations and comments, a lot of conclusions could have been overlooked without them.
Many thanks as well to Bartosz Walczak for his meticulous review.

\pagebreak[3]

\printbibliography

\end{document}

%% file: intro.tex
We consider three interrelated families of graph operations $\PathLF{k}$, $\PathCF{k}$, $\PathRF{k}$, parameterized by an odd integer $k$.
The left operation, 
the \emph{graph $k$-subdivision} $\PathLF{k}(G)$ of a graph $G$, is obtained by replacing every edge with a path on $k$~edges (this is sometimes denoted $G^{\frac{1}{k}}$).
The central operation, 
the $k$-th \emph{power} $\PathCF{k}(G)$ of $G$ is the graph on the same vertex set $V(G)$, with two vertices joined by an edge if they were connected by a walk of length $k$ in $G$ (equivalently, the adjacency matrix is taken to the $k$-th power; this is sometimes denoted $G^{k}$, note however this is \emph{not} the same as joining vertices at distance at most $k$).
Our results concern the right operation, $\PathRF{k}$, which is a certain inverse to the $k$-th power $\PathCF{k}$, as we shall now make precise. 

Let us write $G \to H$ if there exists a graph homomorphism from $G$ to $H$.
Each operation in the above families is a \emph{functor} in the (thin) category of graphs, which means simply that $G \to H$ implies $\SomeF(G) \to \SomeF(H)$, for any graphs $G,H$ (for $\Pi = \PathLF{k}, \PathCF{k}, \PathRF{k}$ with $k$ odd).\footnote{In this paper, we are only concerned with the existence of homomorphisms and maps, not with their identity (compositions, uniqueness). Thus we only consider the \emph{thin} category of graphs (where all homomorphisms $G\to H$ are identified as one arrow), or equivalently, the poset of graphs (with $G\leq H$ when $G \to H$). In the language of posets, functors are just order-preserving maps, while adjoint functors are known as Galois connections. Additional properties required of adjoint functors in the usual (non-thin) categories are not necessarily met, see~\cite{FoniokT17}.\looseness=-1}
More importantly, $\PathCF{k}$~is~a~\emph{right adjoint} to $\PathLF{k}$, meaning that $\PathLF{k}(G) \to H$ holds if and only if $G \to \PathCF{k}(H)$ does.
Similarly, $\PathRF{k}$ is a right adjoint to $\PathCF{k}$, that is, $\PathCF{k}(G) \to H$ if and only if $G \to \PathRF{k}(H)$, though this is less trivial to check.
(This essentially characterizes $\PathRF{k}$, but we give it explicitly with other definitions in Section~\ref{sec:prelims}).
For example, the third power of a graph $G$ admits an $n$-coloring (a homomorphism into the clique $K_n$) if and only if $G \to \Omega_3(K_n)$.

Adjointness of various graph constructions is the principal tool behind Hell and Ne\v{s}et\v{r}il's celebrated theorem (characterizing the complexity of deciding $G\to H$, for fixed $H$)~\cite{HellN90}, in particular the adjointness of $\PathCF{k}$ to $\PathLF{k}$ is used in the first of many steps.
The construction $\PathRF{k}$ was used implicitly by Gy{\'{a}}rf{\'{a}}s et~al.~\cite{GyarfasJS04}, to answer a question on $n$-chromatic graphs with ``strongly independent color classes'': they showed that $\PathRF{3}(K_n)$ gives an example of an $n$-chromatic graph whose third power is still $n$-chromatic.
The construction has also been used in a homomorphism duality theorem by H{\"{a}}ggkvist and Hell~\cite{HaggkvistH93}. 
Tardif~\cite{Tardif05} used iterations of $\PathRF{3}$ and $\PathCF{3}$ to extend results on Hedetniemi's conjecture to the circular chromatic number, showing new multiplicative graphs (we return to this topic later).
Iterating $\PathRF{3}$ $k$ times is equivalent to applying $\PathRF{3^k}$; the operations $\PathRF{k}$ for general odd $k$ can thus be considered just a smoother way to express such iterations.
They were first considered by Hajiabolhassan and Taherkhani~\cite{HajiabolhassanT10}, who proved, among other results, a characterization of the circular chromatic number in terms of the chromatic number (or even just 3-colorability) of powers of graph subdivisions.

The operations $\PathLF{k}$, $\PathCF{k}$, $\PathRF{k}$ are also the simplest example of so called left, central, and right Pultr functors, a more general construction of adjoint graph functors~\cite{pultr1970right}. 
Another simple example of a (central) Pultr functor is the so called \emph{arc graph construction} (see e.g.~\cite{RorabaughTWZ16}), also crucial in applications to Hedetniemi's conjecture~\cite{PoljakR81,tardif2008hedetniemi}.
Graph products and exponential graphs can also be seen as applying Pultr functors.
See~\cite{FoniokT17} for a survey on graph functors focused around Hedetniemi's conjecture and \cite{FoniokT15} for the question of when both left and right adjoints to a common functor exist.


\paragraph*{The topology of graphs}
We show that $\PathRF{k}$ functors also have very good topological properties, with respect to box complexes.
The \emph{box complex} is a construction that assigns a topological space to a graph.
The exact construction is not important for understanding results, but intuitively, $\BBx{G}$ is obtained by taking the graph product $G \times K_2$ as a topological space (edges become copies of the unit interval $[0,1]$), gluing faces to each four-cycle, and similarly gluing higher-dimensional faces to larger complete bipartite subgraphs (see Figure~\ref{fig:K4}).
It originates from Lov{\'{a}}sz's celebrated proof~\cite{Lovasz78} of Kneser's conjecture, as a tool for showing lower bounds on the chromatic number of graphs. Related complexes and applications have since been extensively studied, see~\cite{matousek2008using} for a graceful and extensive introduction to the topic.\looseness=-1
 
The modern view of Lov{\'{a}}sz' technique uses the box complex as a \emph{$\ZZ_2$-space}: a topological space $X$ with a specified \emph{$\ZZ_2$-action} $\nu$ (a homeomorphism $\nu:X\to X$ such that $\nu(\nu(x)) = x$, for $x\in X$).%
\footnote{In this paper, we assume that all $\ZZ_2$-spaces come from a finite simplicial complex (at least up to $\ZZ_2$-homotopy equivalence), unless stated otherwise. We usually omit the $\ZZ_2$-action from the notation.}
A \emph{$\ZZ_2$-map} (also called an equivariant map) between $\ZZ_2$-spaces $(X,\nu_X),(Y,\nu_Y)$ is then a continuous function $f$ from $X$ to $Y$ such that $f(\nu_X(\cdot)) = \nu_Y(f(\cdot))$.
We write $(X,\nu_X) \to_{\ZZ_2} (Y,\nu_Y)$ if any such $\ZZ_2$-map exists (and $\not\to_{\ZZ_2}$ otherwise).
This is a highly non-trivial relation (as opposed to the existence of just continuous maps, since mapping everything into one point would work).
A version of the Borsuk-Ulam theorem says that there is no $\ZZ_2$-map from a higher-dimensional sphere to a lower-dimensional one, $\Sphere^m \not\to_{\ZZ_2} \Sphere^n$ for $m>n$ (the $\ZZ_2$-action of $\Sphere^n$ is implicitly understood to be the antipodal map $x \mapsto (-x)$ on the unit sphere in $\RR^{n+1}$).

The crucial connection is that a homomorphism $G \to H$ induces a $\ZZ_2$-map $\BBx{G} \to_{\ZZ_2} \BBx{H}$.
Since an $n$-coloring of a graph $G$ is the same as a homomorphism from $G$ to $K_n$, and since the box complex of a clique $\BBx{K_{n}}$ can be shown to be an $(n-2)$-dimensional sphere, one concludes that $G$ cannot have an $n$-coloring if $\BBx{G}$ is an $\geq(n-1)$-dimensional sphere, by the Borsuk-Ulam theorem.
This is what allows to get tight lower bounds on the chromatic number of several families of graphs (Kneser, Schrijver, and generalized Mycielski graphs, the graphs $\PathRF{3}(K_n)$ in~\cite{GyarfasJS04}, quadrangulations of the projective plane~\cite{Youngs96} and projective spaces~\cite{KaiserS15}), as well as various results on the circular chromatic number, colorful subgraphs, etc.

\paragraph*{Results}
Our main technical result is that $\PathRF{k}$ functors behave much like subdivision (in the topological sense) on the box complex.
That is, they preserve the homotopy type and they refine the geometric structure, so that any continuous maps between box complexes can be approximated with graph homomorphisms from refinements $\PathRF{k}(G)$ of $G$. See Figure~\ref{fig:K4} for a particularly simple example.
Formally (here $p_k$ is a certain natural homomorphism $\PathRF{k}(G) \to G$, while $|p_k|$ is the induced $\ZZ_2$-map, see Section~\ref{sec:prelims} for definitions):\looseness=-1

\begin{theorem}[Equivalence]\label{thm:equiv}
 $\BBx{G}$ and $\BBx{\PathRF{k}(G)}$ are $\ZZ_2$-homotopy equivalent, for any graph $G$ with no loops and all odd integers $k$.
 Moreover, $|p_k|$ is a $\ZZ_2$-homotopy equivalence.
\end{theorem}

\begin{theorem}[Approximation]\label{thm:approximation}
 There exists a $\ZZ_2$-map from $\BBx{G}$ to $\BBx{H}$ if~and~only~if for some odd $k$, $\PathRF{k}(G)$ has a homomorphism to $H$.

\smallskip
 Moreover, if $G$ and $H$ have no loops, then for every $\ZZ_2$-map $f:\BBx{G} \to_{\ZZ_2} \BBx{H}$, there is an odd~$k$ and a homomorphism $\PathRF{k}(G) \to H$ that induces a map $\ZZ_2$-homotopic to $f \circ |p_k|$.
\end{theorem}

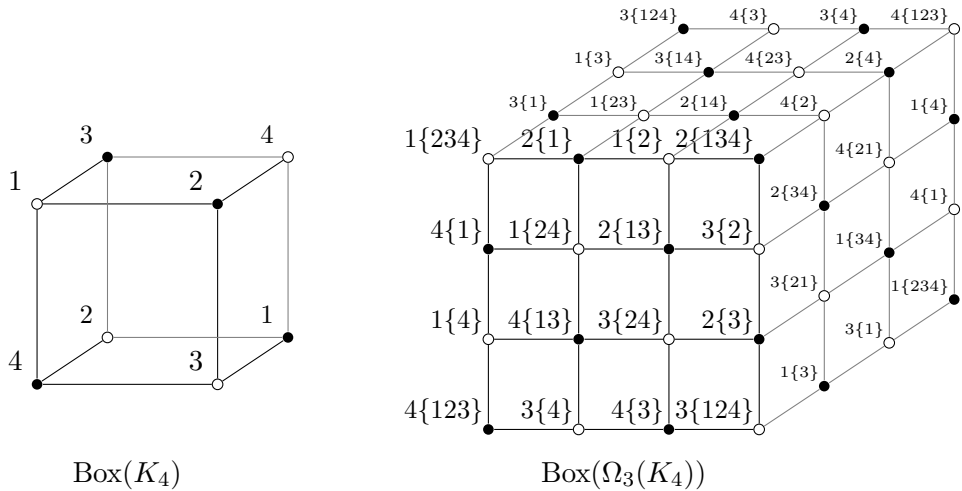
\begin{figure}[b!]
	\centering
	\vspace*{-10pt}
	\input{figK4.tex}
	\vspace*{-5pt}
	\caption{The box complex of the clique graph $K_4$ and of $\PathRF{3}(K_4)$.
	As $G$ becomes $G \times K_2$, for each vertex $v$ of the graph there are two vertices $\vL{v}$ and $\vR{v}$ in the complex.
	Faces are glued to each 4-cycle, making both complexes equivalent to the (hollow) sphere.
	The vertex $(\{v\},\{v_1,v_2,\dots\})$ of $\PathRF{3}(K_4)$ is labeled $v\{v_1 v_2\dots\}$ for short.
	(A careful reader may note that the definition of $\Bx{}$ includes also faces corresponding to each vertex neighborhood, such as the tetrahedron $\{\vL{1},\vR{2},\vR{3},\vR{4}\}$, but it can be shown that these can be collapsed). }
	\vspace*{-2em}
	\label{fig:K4}
\end{figure}

Csorba~\cite{Csorba07} and \v{Z}ivaljevi\'{c}~\cite{Zivaljevic05} independently gave a construction showing that any simplicial complex (with a free $\ZZ_2$-action) is $\ZZ_2$-homotopy equivalent to some box complex (see~\cite{Csorba08} for a simpler proof, see also a generalization to actions of groups other than $\ZZ_2$ and to complexes of homomorphisms in~\cite{Dochtermann12}).

\begin{theorem}[Universality, \cite{Csorba07,Zivaljevic05}]\label{thm:universality}
	For every free $\ZZ_2$-space $X$, there is a graph $G$ such that $X$ and $\BBx{G}$ are $\ZZ_2$-homotopy equivalent.
\end{theorem}

Together, these three theorems show that the homotopy theory of $\ZZ_2$-spaces is largely reflected in graphs, with $\PathRF{k}$ functors as the connection. (Equivalently, in all of our results, `for some odd $k$' can be replaced by `for large enough odd $k$' and  $\PathRF{k}$ by iterations $\PathRF{3}(\dots(\PathRF{3}(G))\dots)$ of $\PathRF{3}$).

The existence of some sequence of functors which satisfy the above Equivalence and Approximation Theorems already follows from the work of Dochtermann and Schultz~\cite[Proposition 4.7]{Dochtermann12}.
Essentially, the idea is to go to the box complex, apply barycentric subdivision iteratively, and return to graphs with the construction from the Universality Theorem.
The construction is however ad-hoc and tedious to describe directly, it cannot be described as iterating a single functor, and it is not clear whether the resulting graph functors admit left adjoints.

One application of the Equivalence Theorem, for $\PathRF{k}$ functors specifically, is that it immediately implies
the result of Gy{\'{a}}rf{\'{a}}s et~al.~\cite{GyarfasJS04}: since the box complex of $\PathRF{k}(K_n)$ has the same homotopy type as that of $K_n$, it is not $(n-1)$-colorable (by the Borsuk-Ulam theorem, as explained above). \looseness=-1
It is then easy to check that it is in fact an $n$-chromatic graph whose $k$-th power is still only $n$-chromatic (in particular it has no loops, so $\PathRF{k}(K_n)$ has no odd cycle of length $\leq k$).
The chromatic number of $\PathRF{k}(K_n)$ (as a ``universal graph for wide colorings'') has also been determined in~\cite{SimonyiT06} and~\cite{BaumS05}.
More generally, for any graph $K$ without loops, $\Omega_k(K)$ gives a graph with a $\ZZ_2$-homotopy equivalent box complex, but with arbitrarily high odd girth (no odd cycles of length $\leq k$).

\subsection*{Hedetniemi's conjecture and multiplicative graphs}
More than 50 year ago, Hedetniemi~\cite{hedetniemi1966homomorphisms} conjectured the following: $\chi(G\times H) = \min(\chi(G),\chi(H))$, for any graphs $G,H$. (Here $\times$ is the \emph{tensor}, or \emph{categorical} product and $\chi$ is the chromatic number, see Section~\ref{sec:prelims} for definitions).
Despite the simplicity of the statement, very little is known: one of the strongest results is a proof by El-Zahar and Sauer~\cite{El-ZaharS85} that the conjecture is true when $G\times H$ is 3-colorable.\footnote{Since the direction $\chi(G\times H) \leq \min(\chi(G),\chi(H))$ is trivial, it follows that the conjecture is also true for all $4$-colorable graphs $G,H$, as the title of~\cite{El-ZaharS85} indicates. This statement may be a bit misleading however, as no insight is gained into the existence of 4-colorings.}
See~\cite{zhu1998survey,Sauer01,tardif2008hedetniemi} for surveys.

The strong connection between graphs and topology, together with the fact that the functors $\PathRF{k}$ commute with the product (which follows from them being right adjoints, see Lemma~\ref{lem:funcproduct}), allows us to show that Hedetniemi's conjecture implies an analogous statement in topology (recall that $\Sphere^d$ denotes the $d$-dimensional sphere with antipodal $\ZZ_2$-action). 
This has recently been shown independently by Matsushita~\cite{Matsushita17}:

\begin{theorem}\label{thm:hedetImplies}
Suppose Hedetniemi's conjecture is true. Then $X \times Y \to_{\ZZ_2} \Sphere^d$ implies $X\to_{\ZZ_2} \Sphere^d$ or $Y\to_{\ZZ_2} \Sphere^d$, for any $\ZZ_2$-spaces $X,Y$ and any integer $d$.
\end{theorem}

Matsushita in fact adapts the box complex construction, the functors of Dochtermann and Schultz~\cite{Dochtermann12}, and the construction of Csorba~\cite{Csorba07}, to give a particularly elegant connection between the category of graphs and the category of $\ZZ_2$-spaces in the form of adjoint functors preserving finite limits, from which the statement readily follows. 
While the approach in this paper does not give such a graceful connection, the author finds it surprising that the most important topological conclusions can also be made using more natural graph functors $\PathRF{k}$, which have already proven to be useful for purely combinatorial theorems. 
Our methods do not give any stronger results than Matsushita's (except maybe for Theorem~\ref{thm:characterization}, where the appearance of $\PathRF{k}$ will make the statement more meaningful as a combinatorial characterization), but we comment more on the implications of Hedetniemi's conjecture and further argue on the importance of topological approaches.

\bigskip
Hedetniemi's conjecture is particularly appealing when $n$-colorings are seen as homomorphisms into $K_n$.
One may more generally consider \emph{multiplicative graphs}:\marginpar{\footnotesize{multipl.\\\hspace*{1em}graph}} a graph $K$ is called multiplicative when $G \times H \to K$ implies $G \to K$ or $H\to K$, for all graphs $G,H$.
Hedetniemi's conjecture is then that all clique graphs $K_n$ are multiplicative.
El-Zahar and Sauer's~\cite{El-ZaharS85} result amounts to saying that $K_3$ is multiplicative.
This has been generalized to odd cycles by H{\"{a}}ggkvist et al.~\cite{HaggkvistHMN88}, to circular cliques $K_{p/q}$ with $p/q<4$ by Tardif~\cite{Tardif05} (using iterations of the $\PathCF{3}$ and $\PathRF{3}$ functors) and to all graphs with no 4-cycles by the author~\cite{Wrochna17} (using the box complex).
A crucial step of Tardif's proof is the fact that a graph $K$ is multiplicative if and only if $\PathRF{3}(K)$ is.

We can analogously\marginpar{\footnotesize{multipl.\\\hspace*{0.5em}$\ZZ_2$-space}} define a $\ZZ_2$-space $Z$ to be \emph{multiplicative} when $X \times Y \to_{\ZZ_2} Z$ implies $X \to_{\ZZ_2} Z$ or $Y\to_{\ZZ_2} Z$, for all $\ZZ_2$-spaces $X,Y$.
Since the box complex of the clique $\BBx{K_n}$ is ($\ZZ_2$-homotopy equivalent to) the $(n-2)$-dimensional sphere $\Sphere^{n-2}$, Theorem~\ref{thm:hedetImplies} is a special case of the following theorem:

\begin{theorem}\label{thm:multImplies}
Let $K$ be a multiplicative graph. Then $\BBx{K}$ is a multiplicative $\ZZ_2$-space.
\end{theorem}

In other words, this means Hedetniemi's conjecture implies the following:

\begin{conjecture}\label{conj:main}
	All spheres $\Sphere^d$ are multiplicative.
\end{conjecture}

We do not know if the converse is true. However, from the multiplicativity of a $\ZZ_2$-space we can deduce a weaker statement, which can be seen as a relaxation of graph multiplicativity, and a combinatorial characterization of multiplicative spaces:
\begin{theorem}\label{thm:characterization}
	Let $Z$ be a $\ZZ_2$-space and let $K$ be a graph such that $\BBx{K} \zeq Z$.
	Then $Z$~is multiplicative if and only if: for all graphs $G,H$,\ $G \times H \to K$ implies that for some odd $k$, $\PathRF{k}(G) \to K$ or $\PathRF{k}(H) \to K$.
\end{theorem}

Thus Conjecture~\ref{conj:main} can be stated as a purely combinatorial statement, relaxing Hedetniemi's conjecture.
However, we note that the conclusion that $\PathRF{k}(G) \to K$ is much weaker than the desired $G \to K$.
For example, circular cliques $K_{p/q}$ with $3 < p/q < 4$ do not admit a homomorphism into $K_3$, but $\PathRF{k}(K_{p/q})$ does (for high enough odd $k$ depending on $p/q$), since the box complex of $K_{p/q}$ is a circle (up to homotopy, in this range of $p/q$).
More strikingly, when $G$ has high girth, then $G$ can have high chromatic number, but $\PathRF{k}(G)$ coincides with the graph $k$-subdivision of $G$ (Lemma~\ref{lem:squarefree}.\ref{l:subdiv}), which is always 3-colorable.

Nevertheless, quite surprisingly, known proofs of multiplicativity for graphs largely follow topological ideas.
In Section~\ref{sec:mult} we give direct, elementary proofs of the multiplicativity of the circle $\Sphere^1$ and discuss the few additional steps needed to conclude the multiplicativity of $K_3$, cycles, and circular cliques.
(We note that Matsushita~\cite{Matsushita17} gives a different, though in essence somewhat similar, direct proof of the multiplicativity of $\Sphere^1$, using the theory of covering spaces).\looseness=-1

This strongly suggests that Conjecture~\ref{conj:main} is crucial to resolving Hedetniemi's conjecture: any counter-example immediately implies a counter-example to Hedetniemi's conjecture, while a proof could be an important first step to a strengthening for graphs (and at least implies a weaker graph-theoretical statement).
Furthermore, a proof of Conjecture~\ref{conj:main} should be in principle easier than any proof of Hedetniemi's conjecture, while obstacles to proving Conjecture~\ref{conj:main} are also obstacles for certain approaches to Hedetniemi's conjecture.
We discuss these in Section~\ref{sec:mult}.

\vspace*{0.5\baselineskip}

\paragraph*{Organization}
Section~\ref{sec:prelims} gives all basic definitions and lists some properties of $\PathLF{k}, \PathCF{k}, \PathRF{k}$ functors.
In Section~\ref{sec:mult} we consider in more detail the implications on Hedetniemi's conjecture and multiplicative graphs; we also comment more on obstacles to generalizations, on Conjecture~\ref{conj:main}, and on open questions that arise.
Section~\ref{sec:multproof} gives the proofs of Theorems~\ref{thm:hedetImplies}, \ref{thm:multImplies} and~\ref{thm:characterization}, which are straightforward applications of the main technical theorems.
The Equivalence Theorem~\ref{thm:equiv} is proved in Section~\ref{sec:equiv} using Discrete Morse Theory, which is also introduced there.
Finally, the Approximation Theorem~\ref{thm:approximation} is proved in Section~\ref{sec:approx}, by considering the geometry of $\BBx{\PathRF{k}(G)}$ and then a fairly standard use of the simplicial approximation technique.

\vspace*{0.5\baselineskip}

%% file: figK4.tex
\begin{tikzpicture}[scale=1.2]
\tikzset{ %
	L/.style={circle,draw=black,fill=white,inner sep=0pt,minimum size=4pt},
	R/.style={circle,fill=black,inner sep=0pt,minimum size=4pt},
	v/.style={circle,draw=black!75,inner sep=0pt,minimum size=8pt},
	h/.style={draw=black!50},
}
\begin{scope}[shift={(-5,0.5)}]
	\newcommand\zscale{1.3}
	\coordinate (z) at ($\zscale*(0.6,0.4)$);
	\node[L,label={135:$1$}] (v1L) at (0,2) {};
	\node[R,label={135:$2$}] (v2R) at (2,2) {};
	\node[L,label={135:$3$}] (v3L) at (2,0) {};
	\node[R,label={135:$4$}] (v4R) at (0,0) {};
	\draw (v1L)--(v2R)--(v3L)--(v4R)--(v1L);
	\node[R,label={135:\small$1$}] (v1R) at ($(2,0)+(z)$) {};
	\node[L,label={135:\small$2$}] (v2L) at ($(0,0)+(z)$) {};
	\node[R,label={135:\small$3$}] (v3R) at ($(0,2)+(z)$) {};
	\node[L,label={135:\small$4$}] (v4L) at ($(2,2)+(z)$) {};
	\draw[h] (v1R)--(v2L)--(v3R)--(v4L)--(v1R);
	\draw (v1L)--(v3R) (v2R)--(v4L) (v3L)--(v1R) (v4R)--(v2L);
	\node at (1,-1) {$\Bx{K_4}$};
\end{scope}

	\node at (1.5,-0.5) {$\Bx{\PathRF{3}(K_4)}$};
	\newcommand\zscale{1.2}
	\coordinate (z) at ($\zscale*(0.6,0.4)$);
	\node[L,label={[label distance=-5pt]135:\small$1\{234\}$}] (v1L) at (0,3) {};
	\node[R,label={[label distance=-5pt]135:\small$2\{134\}$}] (v2R) at (3,3) {};
	\node[L,label={[label distance=-5pt]135:\small$3\{124\}$}] (v3L) at (3,0) {};
	\node[R,label={[label distance=-5pt]135:\small$4\{123\}$}] (v4R) at (0,0) {};
	\node[R,label={[label distance=-5pt]135:\small$2\{1\}$}] (v21R) at (1,3) {};
	\node[L,label={[label distance=-5pt]135:\small$1\{2\}$}] (v12L) at (2,3) {};
	\node[L,label={[label distance=-5pt]135:\small$3\{2\}$}] (v32L) at (3,2) {};
	\node[R,label={[label distance=-5pt]135:\small$2\{3\}$}] (v23R) at (3,1) {};
	\node[R,label={[label distance=-5pt]135:\small$4\{3\}$}] (v43R) at (2,0) {};
	\node[L,label={[label distance=-5pt]135:\small$3\{4\}$}] (v34L) at (1,0) {};
	\node[L,label={[label distance=-5pt]135:\small$1\{4\}$}] (v14L) at (0,1) {};
	\node[R,label={[label distance=-5pt]135:\small$4\{1\}$}] (v41R) at (0,2) {};
	\draw (v1L)--(v21R)--(v12L)--(v2R)--(v32L)--(v23R)--(v3L)--(v43R)--(v34L)--(v4R)--(v14L)--(v41R)--(v1L);
	\node[L,label={[label distance=-5pt]135:\small$1\{24\}$}] (v124L) at (1,2) {};
	\node[R,label={[label distance=-5pt]135:\small$2\{13\}$}] (v213R) at (2,2) {};
	\node[L,label={[label distance=-5pt]135:\small$3\{24\}$}] (v324L) at (2,1) {};
	\node[R,label={[label distance=-5pt]135:\small$4\{13\}$}] (v413R) at (1,1) {};
	\draw (v124L)--(v213R)--(v324L)--(v413R)--(v124L);
	\draw (v41R)--(v124L)--(v21R) (v12L)--(v213R)--(v32L) (v23R)--(v324L)--(v43R) (v34L)--(v413R)--(v14L);
	\node[R,label={[label distance=-5pt]135:\tiny$3\{1\}$}] (v31R) at ($(v1L)+(z)$) {};
	\node[L,label={[label distance=-5pt]135:\tiny$1\{3\}$}] (v13L) at ($(v1L)+2*(z)$) {};
	\node[R,label={[label distance=-5pt]135:\tiny$3\{124\}$}] (v3R) at ($(v1L)+3*(z)$) {};
	\draw[h] (v1L)--(v31R)--(v13L)--(v3R);
	\node[L,label={[label distance=-5pt]135:\tiny$4\{2\}$}] (v42L) at ($(v2R)+(z)$) {};
	\node[R,label={[label distance=-5pt]135:\tiny$2\{4\}$}] (v24R) at ($(v2R)+2*(z)$) {};
	\node[L,label={[label distance=-5pt]135:\tiny$4\{123\}$}] (v4L) at ($(v2R)+3*(z)$) {};
	\draw[h] (v2R)--(v42L)--(v24R)--(v4L);
	\node[R,label={[label distance=-5pt]135:\tiny$1\{3\}$}] (v13R) at ($(v3L)+(z)$) {};
	\node[L,label={[label distance=-5pt]135:\tiny$3\{1\}$}] (v31L) at ($(v3L)+2*(z)$) {};
	\node[R,label={[label distance=-5pt]135:\tiny$1\{234\}$}] (v1R) at ($(v3L)+3*(z)$) {};
	\draw[h] (v3L)--(v13R)--(v31L)--(v1R);
	\node[L,label={[label distance=-5pt]135:\tiny$1\{23\}$}] (v123L) at ($(v21R)+(z)$) {};
	\node[R,label={[label distance=-5pt]135:\tiny$3\{14\}$}] (v314R) at ($(v21R)+2*(z)$) {};
	\node[L,label={[label distance=-5pt]135:\tiny$4\{3\}$}] (v43L) at ($(v21R)+3*(z)$) {};
	\draw[h] (v21R)--(v123L)--(v314R)--(v43L);
	\node[R,label={[label distance=-5pt]135:\tiny$2\{14\}$}] (v214R) at ($(v12L)+(z)$) {};
	\node[L,label={[label distance=-5pt]135:\tiny$4\{23\}$}] (v423L) at ($(v12L)+2*(z)$) {};
	\node[R,label={[label distance=-5pt]135:\tiny$3\{4\}$}] (v34R) at ($(v12L)+3*(z)$) {};
	\draw[h] (v12L)--(v214R)--(v423L)--(v34R);
	\node[L,label={[label distance=-5pt]135:\tiny$3\{21\}$}] (v321L) at ($(v23R)+(z)$) {};
	\node[R,label={[label distance=-5pt]135:\tiny$1\{34\}$}] (v134R) at ($(v23R)+2*(z)$) {};
	\node[L,label={[label distance=-5pt]135:\tiny$4\{1\}$}] (v41L) at ($(v23R)+3*(z)$) {};
	\draw[h] (v23R)--(v321L)--(v134R)--(v41L);
	\node[R,label={[label distance=-5pt]135:\tiny$2\{34\}$}] (v234R) at ($(v32L)+(z)$) {};
	\node[L,label={[label distance=-5pt]135:\tiny$4\{21\}$}] (v421L) at ($(v32L)+2*(z)$) {};
	\node[R,label={[label distance=-5pt]135:\tiny$1\{4\}$}] (v14R) at ($(v32L)+3*(z)$) {};
	\draw[h] (v32L)--(v234R)--(v421L)--(v14R);
	\draw[h] (v31R)--(v123L)--(v214R)--(v42L);
	\draw[h] (v13L)--(v314R)--(v423L)--(v24R);
	\draw[h] (v3R)--(v43L)--(v34R)--(v4L);
	\draw[h] (v13R)--(v321L)--(v234R)--(v42L);
	\draw[h] (v31L)--(v134R)--(v421L)--(v24R);
	\draw[h] (v1R)--(v41L)--(v14R)--(v4L);
\end{tikzpicture}

%% file: prelims.tex
\paragraph*{Graphs}
A \emph{graph} $G$ is a pair $(V(G),E(G))$ where $V(G)$ is a finite set of \emph{vertices} and $E(G)$ is a symmetric relation on vertices.
Instead of $\{u,v\} \in E(G)$ we say that  $uv \in G$ is an edge of $G$ between $u$ and $v$, that $u$ and $v$ are \emph{adjacent}, and that the endpoints $u$ and $v$ are \emph{incident} to the edge $uv$.
Here $u$ may be equal to $v$, such edges are called \emph{loops}.

The \emph{neighborhood} of a vertex $v \in V(G)$,\marginpar{$N(v)$} denoted $N_G(v)$ (where the subscript is omitted when clear from the context) is the set of vertices adjacent to $v$ (it includes $v$ itself if and only if there~is a loop at $v$, that is, $vv \in G$).
For a graph $G$ and two vertex subsets $A,B \subseteq V(G)$,\marginpar{$A\join B$} we write $A \join B$ ($A$ is \emph{joined} to $B$) if all vertices of $A$ are adjacent to all vertices of $B$ (and $A\notjoin B$ otherwise).
We denote the \emph{common neighborhood}\marginpar{$\CN(A)$} of $A\subseteq V(G)$ as $\CN(A) := \bigcap_{v \in A} N(v)$ ($\CN(\emptyset) = V(G)$).
Observe that $A \join B$ if and only if $A \subseteq \CN(B)$ if and only if $B \subseteq \CN(A)$.
Note that $A \join B$ implies that $A$ and $B$ are disjoint, if $G$ has no loops.

The \emph{path} $P_n$ \marginpar{$P_n$\\$C_n$\\$K_n$\\$K_{n,m}$}
is the graph with $V(P_n)=\{1,\dots,n\}$ and $E(P_n)=\{\{i,i+1\} \mid i=1\dots n-1\}$.
The~\emph{cycle} $C_n$ is the graph with $V(C_n)=\ZZ_n$ and $E(C_n)=\{\{i,i+1\} \mid i\in \ZZ_n\}$.
The~\emph{clique} $K_n$ (aka \emph{complete graph}) is the graph with $V(K_n)=\{1,\dots,n\}$ and $E(K_n)=\{\{i,j\} \mid i\neq j = 1\dots n\}$.
The~\emph{biclique} $K_{n,m}$ (aka \emph{complete bipartite graph}) is the graph with $V(K_{n,m})=\{u_1,\dots,u_n,v_1,\dots,v_m\}$ and $E(K_{n,m})=\{\{u_i,v_j\} \mid i=1\dots n,\ j=1\dots m\}$.
The \emph{circular clique}\marginpar{$K_{p/q}$} $K_{p/q}$,
for integers $p,q$ such that $\frac{p}{q}\geq2$, is the graph with $V(K_{p/q})= \ZZ_p$ and $E(K_{p/q}) = \{ \{i,i+j\} \mid j=q,q+1,\dots,p-q,\ i\in\ZZ_p\}$.

A \emph{homomorphism}\marginpar{\footnotesize{homom.}\\$G\to H$} from a graph $G$ to $H$, denoted $f: G \to H$, is a function $f: V(G) \to V(H)$ such that $uv \in E(G)$ implies $f(u)f(v) \in E(H)$, for all $u,v \in V(H)$.
We write $G\to H$ to say that some such homomorphism exists.
Observe that if $H$ has a loop, then every graph has a homomorphism to $H$; if $G$ has a loop, then it can only have a homomorphism to another graph with a loop.
Thus graphs with loops are trivial, for our purposes, but they allow us to formulate statements more uniformly.
Two graphs are \emph{homomorphically equivalent},\marginpar{$G \leftrightarrow H$} denoted $G\leftrightarrow H$, if $G\to H$ and $H\to G$.
An $n$- coloring of $G$ is a homomorphism $G \to K_n$;\marginpar{$\chi(G)$} the \emph{chromatic number} of $G$, denoted $\chi(G)$, is the least $n$ such that $G$ has an $n$-coloring.\marginpar{\footnotesize{$n$-chromatic}}
A graph $G$ is \emph{$n$-chromatic} if $\chi(G)=n$.
The \emph{tensor product}\marginpar{$G \times H$} of graphs $G,H$ (also called the \emph{categorical product}), denoted $G \times H$, is the graph with vertex set $V(G) \times V(H)$, with $(g,h)$ adjacent to $(g',h')$ if and only if $gg' \in G$ and $hh' \in H$.

A\marginpar{\footnotesize{bipartite}}
graph $G$ is \emph{bipartite} if $V(G)$ can be partitioned into two sets $A,B$ such that $E(G)\subseteq \{ \{a,b\} \mid a\in A, b \in B\}$. Equivalently, $G$ contains no odd cycles. Also equivalently, $G\to K_2$.
A \emph{walk}\marginpar{\footnotesize{walk}\\\footnotesize{path}} of length $n$ is a sequence of vertices $v_0,\dots,v_n$ such that $v_i$ is adjacent to $v_{i+1}$ ($i=0\dots n-1$); that is, vertices and edges may repeat, and the length is the number of edges.
A \emph{path} is a walk with no vertex (nor edge) repetitions.
A \emph{subgraph} of a graph $G$ is any graph isomorphic to a graph obtained by removing edges and vertices from $G$.
The subgraph of $G$ \emph{induced} on a subset $S \subseteq V(G)$ is obtained by removing all vertices (with incident edges) outside of $S$.

A graph without loops is \emph{square-free}\marginpar{\footnotesize{square-free}} if it contains no $C_4$ as a subgraph (induced or not).
More generally, a graph $G$ (with loops allowed) is square-free if $A \join B$ implies $|A| \leq 1$ or $|B| \leq 1$ (equivalently, it contains no $C_4$, no two adjacent loops, and no triangle $K_3$ with a looped vertex).

\paragraph*{Topology}
A \emph{simplicial complex} $K$ is a family of non-empty subsets of a finite set, which is downward-closed, that is: $\emptyset \neq \sigma' \subseteq \sigma \in K$ implies $\sigma' \in K$. The elements of $K$ are the \emph{simplices} (or \emph{faces}) of the complex, while the elements of $V(K) := \bigcup_{\sigma \in K} \sigma$ are the \emph{vertices} of the complex.
The\marginpar{$|K|$} \emph{geometric realization} $|\sigma|$ of a simplex $\sigma \in K$ is the subset of $\RR^{V(K)}$ defined as the convex hull of $\{e_v \mid v \in \sigma\}$, where $e_v$ is the standard basis unit vector corresponding to the $v$ coordinate in $\RR^{V(K)}$.
The geometric realization $|K|$ of $K$ is the topological space obtained as the subspace $\bigcup_{\sigma \in X} |\sigma| \subseteq \RR^{V(K)}$.
We often refer to $K$ itself as a topological space, meaning $|K|$.

A $\ZZ_2$-simplicial complex is a simplicial complex $K$ together with a $\ZZ_2$-action: a function $\nu : V(K) \to V(K)$ which maps simplices to simplices and satisfies $\nu(\nu(v)) = v$ for $v \in V(K)$. The $\ZZ_2$-action on $|K|$ is defined by extending $e_v \mapsto e_{\nu(v)}$ linearly on each simplex $|\sigma|$.
We say the\marginpar{\footnotesize{free action}} $\ZZ_2$-action on $K$ is \emph{free}, or $K$ is \emph{free}, if $\{v,\nu(v)\}\not\in K$ for $v\in V(K)$ (equivalently, the image of every simplex $\sigma \in K$, $\nu(\sigma) := \{\nu(v) \mid v \in V(K)\}$, is disjoint from $\sigma$; this implies $\nu(x) \neq x$ for $x \in |K|$).
We\marginpar{$-x, -\sigma$} will usually denote the $\ZZ_2$-action just by $-x$ and $-\sigma$ instead of $\nu(x)$ and $\nu(\sigma)$.

A \emph{$\ZZ_2$-map}\marginpar{\footnotesize{$\ZZ_2$-map}} between two $\ZZ_2$-spaces $X,Y$, denoted $f: X \to_{\ZZ_2} Y$, is a continuous function from $X$ to $Y$ such that $f(\nu_X(x)) = \nu_Y(f(x))$ for $x\in X$.
We write\marginpar{\footnotesize{$X \to_{\ZZ_2} Y$}} $X \to_{\ZZ_2} Y$ if there exists such a map.
A~\emph{$\ZZ_2$-homotopy}\marginpar{\footnotesize{$\ZZ_2$-homot.}} between $f,g: X \to_{\ZZ_2} Y$ is a family of $\ZZ_2$-maps $h_t: X \to_{\ZZ_2} Y$ for $t\in [0,1]$ such that $h_0=f$, $h_1=g$ and such that $(t,x) \mapsto h_t(x)$ is continuous, as a map from $[0,1] \times X$ to $Y$.
We say that $f,g$ are \emph{$\ZZ_2$-homotopic} if such a homotopy exists.
We say that\marginpar{\footnotesize{$X \zeq Y$}} two $\ZZ_2$-spaces $X,Y$ are $\ZZ_2$-homotopy equivalent, denoted $X \zeq Y$, if there are $\ZZ_2$-maps $f: X \to_{\ZZ_2} Y$ and $g: Y \to_{\ZZ_2} X$ such that $g(f(\cdot))$ is $\ZZ_2$-homotopic to the identity on $X$ and $f(g(\cdot))$ is $\ZZ_2$-homotopic to the identity on $Y$.
Both $f$ and $g$ are then called a \emph{$\ZZ_2$-homotopy equivalence}.
Note this is much stronger than just requiring $X \to_{\ZZ_2} Y$ and $Y \to_{\ZZ_2} X$.
The \emph{product}\marginpar{$X \times Y$} $X \times Y$ of two $\ZZ_2$-spaces $X,Y$ is the product space $X,Y$ with $\ZZ_2$-action $(x,y)\mapsto (\nu_X(x),\nu_Y(y))$.
The\marginpar{$\Sphere^n$} $n$-dimensional \emph{sphere} is the $\ZZ_2$-space defined as the unit sphere in $\RR^{n+1}$ with $\ZZ_2$-action $x \mapsto -x$.

\paragraph*{Box  complex} The \emph{box complex}\marginpar{$\Bx{G}$} $\Bx{G}$ of a graph $G$ is a simplicial complex defined as follows.
If $G$ has isolated vertices (vertices with no neighbors), first remove all of them from $G$.
Let the vertex set of $\Bx{G}$ be $V(G) \times \{\circ,\bullet\}$; that is, for every (non-isolated) vertex $v \in V(G)$, the\marginpar{$\vL{v},\vR{v},\vU{v}$} simplicial complex has two vertices, which we denote $\vL{v}$ and $\vR{v}$.
We will also write $\vU{v}$ when $\unk \in \{\circ,\bullet\}$ is clear from the context.
For a set $A \subseteq V(G)$, we write \marginpar{$\vL{A},\vR{A}$}  $\vL{A}$ and $\vR{A}$ for $\{\vL{v} \mid v \in A\}$ and $\{\vR{v} \mid v \in A\}$.
For\marginpar{$\vvL{\sigma},\vvR{\sigma}$} a set $\sigma \subseteq V(G) \times \{\circ,\bullet\}$, we write $\vvL{\sigma}$ and $\vvR{\sigma}$ for $\{v \mid \vL{v} \in \sigma\}$ and $\{v \mid \vR{v} \in \sigma\}$, respectively.
(To avoid confusion, we denote simplices with small Greek letters and vertex subsets with capital Latin letters).
The simplices of $\Bx{G}$ are exactly those sets $\sigma \subseteq V(G) \times \{\circ,\bullet\}$ such that $\vvL{\sigma} \join \vvR{\sigma}$ and both $\CN(\vvL{\sigma})$ and $\CN(\vvR{\sigma})$ are non-empty (in other words, the non-trivial complete bipartite subgraphs of $G$ and their subsets).
Note that if $\vvL{\sigma} \neq \emptyset$, then $(\vvL{\sigma} \times\{\circ\}) \cup (\CN(\vvL{\sigma}) \times \{\bullet\})$ is again a simplex, containing $\sigma$; similarly for $\vvR{\sigma}$; hence all maximal simplices $\sigma$ have both $\vvL{\sigma}$ and $\vvR{\sigma}$ non-empty.
The $\ZZ_2$-action $-$ on $\Bx{G}$ is defined as  $-\vL{v}=\vR{v}$ and $-\vR{v}=\vL{v}$ for each $v \in V(G)$.
If $G$ has no loops, then $\vL{v}$ is never in a simplex together with $\vR{v}$, so the $\ZZ_2$-action is free (the converse is also true).
As\marginpar{$|h|$} mentioned in the introduction, a homomorphism $h: G \to H$ induces a $\ZZ_2$-map $\BBx{G} \to_{\ZZ_2} \BBx{H}$, which we denote $|h|$ (thus one can also think of it as a functor, into the category of $\ZZ_2$-spaces and $\ZZ_2$-maps).

\paragraph*{Graph functors}
An\marginpar{\footnotesize{functor,}} operation $\CentralF$ on graphs is a (thin) \emph{functor} if $G\to H$ implies $\CentralF(G) \to \CentralF(H)$, for all graphs $G,H$.
Two\marginpar{\footnotesize{adjoint}} functors $\CentralF,\RightF$ are called a (thin) \emph{adjoint pair} when $\CentralF(G) \to H$ holds if and only if $G \to \RightF(H)$ does. In this case $\CentralF,\RightF$ are called \emph{left} and \emph{right} adjoints, respectively.
Note that a right adjoint functor may be a left adjoint in another pair, as is the case for the $\PathCF{k}$ functor.
The graph subdivision functor $\PathLF{k}$ and powering functor $\PathCF{k}$ were defined in the introduction. 

For\marginpar{$\PathRF{k}(G)$} a graph $G$ and an integer $\ell$, the graph $\PathRF{2\ell+1}(G)$ is defined as follows.
Its vertices are tuples $\tup{A}=(A_0,\dots,A_\ell)$ of vertex subsets $A_i \subseteq V(G)$ such that $A_0$ is a singleton (contains exactly one vertex) and $A_{i-1} \join A_{i}$ (for $i=1\dots \ell$).
Its edges are pairs $\{\tup{A}, \tup{B}\}$ such that $A_{i-1} \subseteq B_i$, $B_{i-1} \subseteq A_i$, and $A_\ell \join B_\ell$ (note this implies $A_i \join B_i$ for $i=0\dots \ell$).
We\marginpar{$p_k$} define the homomorphism $p_{2\ell+1}\colon\PathRF{2\ell+1}(G) \to G$ as $p_{2\ell+1}((\{v\},A_1,\dots,A_\ell)) := v$.
We do not define $\PathRF{k}(G)$ for even integers $k$ (see~\cite{FoniokT17} for a functor $\PathRF{2}$ that shares some properties).

\medskip
We now list a few basic properties of these functors.
For the box complex, an important property is that $\Bx{}$ commutes with products:\footnote{As noted in \cite{SimonyiZ10}, Remark 3, this follows from Theorem~9 in~\cite{Csorba08} saying that $\BBx{H}$ is $\ZZ_2$-homotopy equivalent to the so-called \emph{hom complex} $\Hom(K_2,G)$ and from the fact that $\Hom(K_2,G \times H)$ is isomorphic, as a $\ZZ_2$-simplicial complex, to $\Hom(K_2, G) \times \Hom(K_2, H)$. The isomorphism is tedious but straightforward to check from definitions.
The small additional complexity in existing proofs, which show only homotopy equivalence instead of isomorphism, is due to considering posets instead of simplicial complexes.}

\begin{lemma}[\cite{Csorba08}]\label{lem:boxproduct}
	$\BBx{G} \times \BBx{H} \zeq \BBx{G \times H}$.
\end{lemma}

Similarly, any right adjoint graph functor commutes with the tensor product.
We state this together with a few other simple properties.
The proofs are straightforward. The applications to multiplicativity (Lemma~\ref{lem:funcproduct}.\ref{l:rightmult} and~\ref{lem:squarefree}.\ref{l:mult}) have first been shown and used by Tardif~\cite{Tardif05}; we do not use them in this paper and recall them only for reference.

\pagebreak[3]

\begin{lemma}\label{lem:funcproduct}
	Let $G,G_1,G_2$ be any graphs. Then:
	\begin{enumerate}[label=(\roman*)]
		\item\label{l:proj} $G_1 \times G_2 \to G_i$, for $i=1,2$;
		\item\label{l:and} $G \to G_1 \times G_2$ if and only if $G \to G_1$ and $G \to G_2$;
		\item\label{l:funcprod} if $\CentralF$ is a functor, then $\CentralF(G_1 \times G_2) \to \CentralF(G_1) \times \CentralF(G_2)$;
		\item\label{l:adjoint} if $(\CentralF,\RightF)$ is an adjoint pair of functors, then $\CentralF(\RightF(G)) \to G \to \RightF(\CentralF(G))$;
		\item\label{l:rightprod} if $\RightF$ is a right adjoint, then $\RightF(G_1 \times G_2) \leftrightarrow \RightF(G_1) \times \RightF(G_2)$;
		\item\label{l:rightmult} if $\RightF$ is a right adjoint to a functor that is a right adjoint itself, and if $K$ is a multiplicative graph, then $\RightF(K)$ is multiplicative too.
	\end{enumerate}
\end{lemma}

\noindent
We follow with a few properties more specific to $\PathLF{k}$, $\PathCF{k}$, and $\PathRF{k}$.
Most of these have been shown by Tardif~\cite{Tardif05} or by Hajiabolhassan and Taherkhani~\cite{Hajiabolhassan09, HajiabolhassanT10}, who also proved many properties of other compositions of these functors (which can be interpreted as ``fractional powers'').
As far as we know, \ref{l:subgr} and \ref{l:subdiv} are folklore, but have not appeared earlier in the literature.

\begin{lemma}\label{lem:squarefree}
	Let $G,H,K$ be graphs and let $k,k'$ be odd integers. Then:
	\begin{enumerate}[label=(\roman*)]
		\item\label{l:adj} $\PathLF{k}(G) \to H$ if and only if $G \to \PathCF{k}(H)$ (that is, $(\PathLF{k},\PathCF{k})$ is an adjoint pair);
		\item \label{l:adj2} $\PathCF{k}(G) \to H$ if and only if $G \to \PathRF{k}(H)$ (that is, $(\PathCF{k},\PathRF{k})$ is an adjoint pair);		
		\item\label{l:downsub} $\PathLF{k}(G) \to \PathLF{k-2}(G) \to \dots \to \PathLF{1}(G) = G = \PathCF{1}(G) \to \dots \to \PathCF{k-2}(G) \to \PathCF{k}(G)$;
		\item\label{l:down} $\PathRF{k}(G) \to \PathRF{k-2}(G) \to \dots \to \PathRF{1}(G) = G$;
		\item\label{l:comp} $\PathLF{k}(\PathLF{k'}(G)) \leftrightarrow \PathLF{k\cdot k'}(G)$,\ \  $\PathCF{k}(\PathCF{k'}(G)) \leftrightarrow \PathCF{k\cdot k'}(G)$,\ \ and\ \ $\PathRF{k}(\PathRF{k'}(G)) \leftrightarrow \PathRF{k\cdot k'}(G)$;
		\item\label{l:subgr} $\PathLF{k}(G) \subseteq \PathRF{k}(G)$, in particular $\PathLF{k}(G) \to \PathRF{k}(G)$;
		\item\label{l:subdiv} if $G$ is square-free, then $\PathLF{k}(G) \leftrightarrow \PathRF{k}(G)$;
		\item \label{l:iff2} $\PathCF{k}(\PathRF{k}(G)) \leftrightarrow G \leftrightarrow \PathCF{k}(\PathLF{k}(G))$;
		\item \label{l:iff} $G \to H$ if and only if $\PathRF{k}(G) \to \PathRF{k}(H)$;
		\item\label{l:mult} $K$ is multiplicative if and only if $\PathRF{k}(K)$ is.
	\end{enumerate}
\end{lemma}
\begin{proof}
	Let $k=2\ell+1$.
	\ref{l:adj} and \ref{l:downsub} follow straight from definitions.
	For one direction of \ref{l:adj2}, let $f: \PathCF{k}(G) \to H$;
	then a homomorphism $G \to \PathRF{k}(H)$ is given by $v \mapsto (f(N^0(v)),\dots,f(N^\ell(v)))$ where $N^i(v)$ is the set of vertices reachable from $v$ by walks of length $i$, as one can easily check.
	For the other direction, let $f: G \to \PathRF{k}(H)$;
	then a homomorphism $\PathCF{k}(G) \to H$ is given by mapping $v$ to the only vertex in the first, singleton set of $f(v) = (A_0,\dots,A_\ell)$.	
	
	For \ref{l:down}, $(A_0,\dots,A_{\ell-1},A_\ell) \mapsto (A_0,\dots,A_{\ell-1})$ gives the homomorphism (where $k=2\ell+1$).
	For \ref{l:comp} observe that $\PathLF{k}(\PathLF{k'}(G)) \leftrightarrow \PathLF{k\cdot k'}(G)$ follows from the definition. Then since $\PathCF{k \cdot k'}$ is a right adjoint to $\PathLF{k\cdot k'}$, which is homomorphically equivalent (when applied to any graph) to  $\PathLF{k}(\PathLF{k'}(\cdot))$, which in turn is a left adjoint to $\PathCF{k'}(\PathCF{k}(\cdot))$, it follows that $\PathCF{k \cdot k'}(G) \leftrightarrow \PathCF{k'}(\PathCF{k}(G))$ (for all $G,k,k'$).
	Similarly for $\PathRF{k}$.

	For \ref{l:subgr}, let us define the following injective homomorphism $\PathLF{k}(G) \to \PathRF{k}(G)$.
	For $ab\in E(G)$, the path of length $k$ between $a$ and $b$ in $\PathLF{k}(G)$ is mapped to the following path in $\PathRF{k}(G)$.
	(The two vertices in the middle should be swapped when $\lfloor \frac{k}{2}\rfloor$ is even).
	It is straightforward to check this defines an injective homomorphism in a consistent way:
	
	\pagebreak[3]
	
	\vspace*{-0.7\baselineskip}
	$$\def\arraystretch{0.97}\begin{array}{llllllll}
	&(\{a\},& N(a),&\{a\},&N(a),&\{a\},&\dots),\\
	&(\{b\},& \{a\},&N(a),&\{a\},&N(a),&\dots),\\
	&(\{a\},& \{b\},&\{a\},&N(a),&\{a\},&\dots),\\
	&\quad\vdots&&&&\\
	&(\{b\},& \{a\},&\{b\},&\{a\},&\{b\},&\dots),&\ (\leftarrow)\\
	&(\{a\},& \{b\},&\{a\},&\{b\},&\{a\},&\dots),&\ (\leftarrow\text{or vice-versa})\\
	&\quad\vdots&&&&\\
	&(\{b\},& \{a\},&\{b\},&N(b),&\{b\},&\dots),\\
	&(\{a\},& \{b\},&N(b),&\{b\},&N(b),&\dots),\\
	&(\{b\},& N(b),&\{b\},&N(b),&\{b\},&\dots).\\
	\end{array}$$

	\pagebreak[3]
	To show \ref{l:subdiv}, let $k=2\ell+1$. We construct $f:\PathRF{2\ell+1}(G) \to \PathLF{2\ell+1}(G)$ as follows.
	For $\tup{A}=(A_0,\dots,A_\ell) \in V(\PathRF{2\ell+1}(G))$ with $A_0=\{a\}$, let $j_{\tup{A}}$ be the maximum index such that $A_i$ are singletons for $i\leq j_{\tup{A}}$.
	If $j_{\tup{A}} = 0$, we set $f(\tup{A})=a$, otherwise let $A_1=\{b\}$ and we set $f(\tup{A})$ to be the $i$-th vertex on the path between $a$ and $b$ (counting $a$ as the 0-th vertex), where $i=j_{\tup{A}}$ if~$j_{\tup{A}}$ is even, while $i=2\ell+1 - j_{\tup{A}}$ if $j_{\tup{A}}$ is odd.
	\pagebreak[3]

	Let $\tup{A},\tup{B}$ be adjacent in $\PathRF{2\ell+1}(G)$.
	Since $A_\ell \join B_\ell$ and $G$ is square-free, one of $A_\ell,B_\ell$ must be of size at most $1$.
	Assume without loss of generality that $|A_\ell|\leq 1$ (otherwise swap $\tup{A}$ and $\tup{B}$).
	Since $A_\ell \supseteq B_{\ell-1} \supseteq A_{\ell-2} \supseteq \dots$ is a sequence of containments ending in a singleton $A_0$ or $B_0$, all these containments are equalities.
	Let us also assume that $\ell$ is odd (the proof is the same with $\ell$ even).
	That is, the sequence ends in $B_0$ and $A_\ell=B_{\ell-1}=A_{\ell-2}=\dots=B_0$ is a singleton.
	Let $B_0 = \{b\}$ and $A_0=\{a\}$.
	Consider the sequence $A_0 \subseteq B_1 \subseteq A_2 \subseteq \dots \subseteq B_\ell$ and let $j$ be the maximum index such that the $j$-th set of this sequence is a singleton, and hence equal to $A_0=\{a\}$, as well as to all the sets in between.
	Then, since the next sets in the sequence (if there are any) are not singletons, we have $j_{\tup{A}}=j$ and $j_{\tup{B}}=j+1$ or vice versa (depending on the parity of $j$), unless $j=\ell$, in which case $j_{\tup{A}}=j_{\tup{B}}=\ell$.
	It each case, it is easily checked that $f(\tup{A})$ and $f(\tup{B})$ are adjacent in $\PathLF{2\ell+1}(G)$.

	Observe that \ref{l:subgr} implies $\PathLF{k}(G) \to \PathRF{k}(G)$, which by adjointness is equivalent to
	$G \to \PathCF{k}(\PathRF{k}(G))$ and to $\PathCF{k}(\PathLF{k}(G)) \to G$. This, together with Lemma~\ref{lem:funcproduct}.\ref{l:adjoint}, implies \ref{l:iff2}.
	Applying $\PathCF{k}$ to both sides of the assumption $\PathRF{k}(G) \to \PathRF{k}(H)$ thus yields the non-trivial direction of \ref{l:iff}.

	For \ref{l:mult}, one direction follows from Lemma~\ref{lem:funcproduct}.\ref{l:rightmult}.
	For the other, suppose $\PathRF{k}(K)$ is multiplicative.
	Let $G \times H \to K$. Then $\PathRF{k}(G) \times \PathRF{k}(H) \to \PathRF{k}(G \times H) \to \PathRF{k}(K)$,
	hence $\PathRF{k}(G) \to \PathRF{k}(K)$ or $\PathRF{k}(H) \to \PathRF{k}(K)$,
	which by~\ref{l:iff} implies $G \to K$ or $H \to K$.
\end{proof}

%% file: mult.tex
\subsection*{Known cases of multiplicativity}
As a warm-up, let us give an elementary proof of the multiplicativity of the 0-dimensional sphere $\Sphere^0$ (two points $-1,1$ on the real line, with the $\ZZ_2$-action swapping them), in the following two lemmas.
(The first will be crucial to the multiplicativity of $\Sphere^1$ as well).

\begin{lemma}\label{lem:S0S1}
	Let $X$ be a $\ZZ_2$-space.
	Then $\Sphere^1 \not\to_{\ZZ_2} X$ if and only if $X \to_{\ZZ_2} \Sphere^0$.
\end{lemma}
\begin{proof}
	Let $p \colon \Sphere^1 \to_{\ZZ_2} X$ be a $\ZZ_2$-map.
	Then $p$ on one half of $\Sphere^1$ gives a path $p'\colon [0,1] \to X$ from some point $p(0)=x \in X$ to $p(1)=-x$.
	If there was a map $f\colon X \to_{\ZZ_2} \Sphere^0$, then each (path-)connected component of $X$ would have to map all into $-1$ or all into $1\in \Sphere^0$, in particular $f(x)=f(-x)$, a contradiction.

	For the other direction, assume $\Sphere^1 \not\to_{\ZZ_2} X$.
	Then there is no path $p\colon [0,1] \to X$ from a point $x \in X$ to $-x$, since concatenating such a path $t \mapsto p(t)$ with $t \mapsto -p(t)$ gives a $\ZZ_2$-map $\Sphere^1 \to_{\ZZ_2} X$.
	Therefore, the $\ZZ_2$-action $-$ matches pairs of (path-)connected components of $X$.
	We can choose a map that maps one component of each pair into $-1$ and the other into $1$, giving a $\ZZ_2$-map $X \to_{\ZZ_2} \Sphere^0$.
\end{proof}

To translate the above proof to graphs, recall that the antipode of $\vL{v}$ in the box complex of a graph $G$ (for a vertex $v$ of $G$) is $\vR{v}$ and observe that there is a path from $\vL{v}$ to $\vR{v}$ in the box complex if and only if there is a walk of odd length in the graph $G$ from $v$ to $v$ itself.
That is, `equivariant' circles in the box complex, represented as $\ZZ_2$-maps $\Sphere^1 \to_{\ZZ_2} X$, correspond to odd closed walks in the graph.
The above lemma then corresponds to the fact that a graph has no odd closed walks (equivalently, no odd cycles) if and only if it has a homomorphism to $K_2$ (equivalently, it is bipartite).
The proof can also be made entirely analogous, by considering connected components of $G \times K_2$.
We proceed with a proof of multiplicativity.

\begin{lemma}\label{lem:S0}
$\Sphere^0$ is multiplicative.
That is, for any $\ZZ_2$-spaces $X,Y$, if $X \times Y \to_{\ZZ_2} \Sphere^0$, then $X\to_{\ZZ_2} \Sphere^0$ or $Y\to_{\ZZ_2} \Sphere^0$.
\end{lemma}
\begin{proof}
	Suppose that $X\not\to_{\ZZ_2} \Sphere^0$  and $Y\not\to_{\ZZ_2} \Sphere^0$.
	By Lemma~\ref{lem:S0S1}, there are $\ZZ_2$-maps $p\colon \Sphere^1 \to_{\ZZ_2} X$ and $q\colon \Sphere^1 \to_{\ZZ_2} Y$.
	But then $t \mapsto (p(t),q(t))$ is a $\ZZ_2$-map $\Sphere^1 \to_{\ZZ_2} X \times Y$ (since by definition of the product of $\ZZ_2$-spaces, $-(p(t),q(t))=(-p(t),-q(t))=(p(-t),q(-t))$).
	Hence $X \times Y \not\to_{\ZZ_2} \Sphere^0$.
\end{proof}

The multiplicativity of $K_2$ is  a simple translation of this proof: if $G \not\to K_2$ and $H \not\to K_2$, then there are odd closed walks $P=(P_1,\dots,P_n)$ ($P_i \in V(G)$) in $G$ and $Q=(Q_1,\dots,Q_m)$ in $H$. 
We can turn them into odd closed walks of equal length, say $P'=P$ and $Q'=(Q_1,\dots,Q_m,Q_{m-1},Q_m,\dots,Q_{m-1},Q_m)$ if $n\leq m$.
Thus $((P'_1,Q'_1),(P'_2,Q'_2),\dots)$ is an odd closed walk in $G \times H$, hence $G \times H \not\to K_2$.

\medskip
After this warm-up, let us turn to the circle $\Sphere^1$.
We first show that in the definition of multiplicativity, we can assume spaces to be path-connected.
\begin{lemma}\label{lem:conn}
	Let $Z$ be a non-empty $\ZZ_2$-space such that for any path-connected $\ZZ_2$-spaces $X,Y$,
	$X\times Y \to_{\ZZ_2} Z$ implies $X\to_{\ZZ_2} Z$ or $Y\to_{\ZZ_2} Z$.
	Then $Z$ is multiplicative.
\end{lemma}
\begin{proof}
	The $\ZZ_2$-action on $X$ maps each (path-connected) component into a component.
	Let us call components that are mapped to themselves \emph{$\ZZ_2$-components}.
	Components that are not mapped to themselves are matched in pairs by the $\ZZ_2$-action.
	Every such pair of components can be trivially $\ZZ_2$-mapped into $Z$ by mapping all of one component into one arbitrary point of $Z$ and all of the other component into the $\ZZ_2$-image of that point.
	Therefore, $X$ admits a $\ZZ_2$-map to $Z$ if and only if each $\ZZ_2$-component of $X$ does.
	Suppose $X\not\to_{\ZZ_2} Z$  and $Y\not\to_{\ZZ_2} Z$, that is, 
	some $\ZZ_2$-components $X'$ of $X$ and $Y'$ of $Y$ do not admit a $\ZZ_2$-map to $Z$.
	By the assumption on $Z$, $X \times Y \supseteq X' \times Y' \not\to_{\ZZ_2} Z$.
	Hence $Z$ is multiplicative.
\end{proof}

\begin{lemma}\label{lem:S1}
$\Sphere^1$ is multiplicative.
That is, for any $\ZZ_2$-spaces $X,Y$, if $X \times Y \to_{\ZZ_2} \Sphere^1$, then $X\to_{\ZZ_2} \Sphere^1$ or $Y\to_{\ZZ_2} \Sphere^1$.
\end{lemma}
\begin{proof}
	Let $f \colon X \times Y \to_{\ZZ_2} \Sphere^1$ be a $\ZZ_2$ map (that is, $-f(x,y)=f(-x,-y)$ for all $x\in X, y\in Y$).
	By Lemma~\ref{lem:conn} we can assume that $X$ and $Y$ are (path-)connected.
	Fix arbitrary points $x_0 \in X, y_0 \in Y$.

	Consider any $\ZZ_2$-maps $p \colon \Sphere^1 \to_{\ZZ_2} X$ and $q\colon \Sphere^1 \to_{\ZZ_2} Y$ starting (and ending) at $x_0$ and $y_0$, respectively.
	For $t \in \Sphere^1$, let $p'(t) := f(p(t),y_0)$ and $q'(t) := f(x_0,q(t))$.
	The functions $p',q'$ are continuous maps from $\Sphere^1$ to $\Sphere^1$ (not necessarily $\ZZ_2$-maps).
	Since the concatenation of the path $p(t)$ with the constant path $t \mapsto x_0$ is homotopic to $p(t)$ (and similarly for $q(t)$ and $y_0$), the concatenation of the paths $t \mapsto (p(t), y_0)$ and $t\mapsto (x_0,q(t))$ in $X \times Y$ is homotopic to $t \mapsto (p(t),q(t))$.
	Therefore the concatenation of $p'(t)=f(p(t),y_0)$ and $q'(t) = f(x_0,q(t))$ is homotopic to $t \mapsto f(p(t),q(t))$.
	Thus the winding numbers of $p'$ and $q'$ sum to the winding number of $t \mapsto f(p(t),q(t))$.
	The latter is a $\ZZ_2$-map (because $f,p,q$ are) and hence has odd winding number.
	Therefore exactly one of the winding numbers of $p'$ and $q'$ is odd.
	Without loss of generality suppose the winding number of $p'$ is odd and the winding number of $q'$ is even.
	Then the winding number of $p'$ is odd for any choice of  $p \colon \Sphere^1 \to_{\ZZ_2} X$ starting and ending at $x_0$, as we can keep the choice of $q,q'$ unchanged (with even winding number).

	The same holds for every $p \colon \Sphere^1 \to_{\ZZ_2} X$ which does not start and end at $x_0$, since $p$ is then homotopic to a $\ZZ_2$-map that does. (Namely, by path-connectedness of $X$, there is some path $s$ from the endpoint $x$ of $p$ to $x_0$;
	let~$r$ denote the first half of $p$ -- a path from $x$ to $-x$; let $s^{-1}$ denote the path $s$ reversed and let $-s$ be its $\ZZ_2$-image; then $p$ is homotopic to the concatenation of $s^{-1}, r, -s, -s^{-1}, -r, s$). The homotopy preserves the (odd) winding number of $p'$.
	
	For any $p \colon \Sphere^1 \to_{\ZZ_2} X$,
	since the winding number of $p'$ is odd, there is a point $t \in \Sphere^1$ such that $p'(-t) = -p'(t)$.
	That is $f(p(-t),y_0) = -f(p(t),y_0)$.
	Let us call a point $x\in X$ a \emph{coincidence point} if $f(x,y_0) = -f(-x,y_0)$ (equivalently, $f(x,y_0) = f(x,-y_0)$).
	Let $X' \subseteq X$ be the set of coincidence points (observe that if $x \in X'$, then $-x \in X'$ as well).
	Then we know that there is no $\ZZ_2$-map $p \colon \Sphere^1 \to_{\ZZ_2} X \setminus X'$.
	Therefore, there is a $\ZZ_2$-map $h\colon X \setminus X' \to \Sphere^0$.
	We can then define a $\ZZ_2$-map from $X$ to $\Sphere^1$ as follows:
		if $x \in X\setminus X'$, we map $x$ to $f(x,-y_0)$ or to $f(x,y_0)$ depending on $h(x)\in \{-1,1\}$;
		otherwise, if $x \in X'$, we map $x$ to $f(x,y_0)=f(x,-y_0)$.
	This is easily checked to give a $\ZZ_2$-map from $X$ to $\Sphere^1$.
\end{proof}

The proof of the multiplicativity of $K_3$ by El-Zahar and~Sauer~\cite{El-ZaharS85}, its generalization to odd cycles by H{\"{a}}ggkvist et~al.~\cite{HaggkvistHMN88}, and especially its reformulation and generalization to  circular cliques $K_{p/q}$ (with $2<p/q<4$) given in~\cite{Wrochna17}, largely follows the steps of the above proof of Lemma~\ref{lem:S1}.
An invariant on odd cycles is considered, which turns out to be exactly the winding number assigned as above to the corresponding map $\Sphere^1 \to_{\ZZ_2} \BBx{G}$.
One then proves that all odd cycles on one side of the product must have an odd invariant, which implies that certain coincidence points must exist on every such cycle (this part can be done just as above, purely topologically, in the box complex).
If those coincidence points occur on vertices of the box complex (corresponding to vertices of the graph), as opposed to some general position on edges or larger simplices, then they can be temporarily removed to conclude a homomorphism just as above. 
The only additional step is thus showing that the coincidence points can be assumed to lie on vertices, which is indeed true and not hard to show for $K_3$ and odd cycles.
However, for circular cliques a certain relaxation of this notion is necessary (but still possible, see~\cite{Wrochna17}), while for other graphs $G$ with $\BBx{G} \zeq \Sphere^1$ we do not know whether this approach can work at all, indeed we do not know whether all such graphs are multiplicative.

\medskip
The multiplicativity of square-free graphs shown in~\cite{Wrochna17} corresponds to the multiplicativity of 1-dimensional $\ZZ_2$-spaces with free $\ZZ_2$-action, that is, those coming from simplicial complexes with no simplices of size larger than~2.
However, all such spaces can be shown to admit a $\ZZ_2$-map to $\Sphere^1$ (by taking a spanning tree of the quotient by the $\ZZ_2$-action, contracting its two preimages to points and mapping the remaining edges appropriately).
Because of that, their multiplicativity easily follows from that of $\Sphere^1$. (Indeed, either they admit a $\ZZ_2$-map \emph{from} $\Sphere^1$ as well, in which case they are essentially equivalent to $\Sphere^1$, as far as multiplicativity is concerned; or they do not, in which case they admit a homomorphism to $\Sphere^0$ by Lemma~\ref{lem:S0S1} and from $\Sphere^0$, unless empty).
This reasoning does not extend to the combinatorial setting, unfortunately.
The proof in~\cite{Wrochna17} instead uses some stronger topological properties, essentially allowing to lift a map to a covering space which is then contracted.
The translation to graphs, using homomorphisms instead of on $\ZZ_2$-maps, is considerably more technical than for $\Sphere^1$.

\subsection*{Obstacles and non-tidy spaces}
When attempting to generalize the above proofs to higher dimensional spheres, even just to~$\Sphere^2$, while some steps do extend (the arguments on the parity of winding numbers, in particular), there are nevertheless substantial obstacles.
Perhaps the most important is the fact that Lemma~\ref{lem:S0S1} becomes false: there are $\ZZ_2$-spaces $X$ such that $\Sphere^2 \not\to_{\ZZ_2} X$, but $X \not\to_{\ZZ_2} \Sphere^1$.

This gap can in fact get much worse.
Consider the following two parameters of a $\ZZ_2$-space $X$.
The \emph{coindex}\marginpar{$\ind$\\$\coind$} $\coind(X)$ is the largest $n$ such that $\Sphere^n \to_{\ZZ_2} X$.
The \emph{index} is the least $n$ such that $X \to_{\ZZ_2} \Sphere^n$.
The Borsuk-Ulam theorem states that $\coind(X) \leq \ind(X)$.
These parameters are analogous to the clique number $\omega(G)$ (the size of the largest clique subgraph) and the chromatic number $\chi(G)$ of a graph $G$. In fact:
\vspace*{-9pt}
$$ \omega(G)\ \leq\ \coind(\BBx{G}) + 2\ \leq\ \ind(\BBx{G}) + 2\ \leq\ \chi(G).\vspace*{-5pt}$$

Spaces where the coindex is strictly smaller than the index are called \emph{non-tidy} (see~\cite{matousek2008using}, p. 100).
Lemma~\ref{lem:S0S1} states that the coindex is 0 if and only if the index is 0, so non-tidy spaces are counter-examples to its generalization, and thus a significant problem when attempting to extend known cases of Hedetniemi's conjecture. 
Moreover, Conjecture~\ref{conj:main} is equivalent to the statement that, for all $\ZZ_2$-spaces $X,Y$:
\vspace*{-9pt}
$$\ind(X \times Y) = \min(\ind(X),\ind(Y)).\vspace*{-5pt}$$
Since the inequality $\ind(X \times Y) \leq \min(\ind(X), \ind(Y))$ is trivial and since $X \to_{\ZZ_2} Y$ easily implies the other direction, any counter-example to Conjecture~\ref{conj:main} must satisfy $X \not\to_{\ZZ_2} Y$ and $Y \not\to_{\ZZ_2} X$.
Since $\coind(X) \geq \ind(Y)$ implies $Y \to_{\ZZ_2} \Sphere^{\ind(Y)} \to_{\ZZ_2} X$, any counter-example to Conjecture~\ref{conj:main} must involve a non-tidy space. 

Non-tidy spaces are not so easy to come by, at least for a combinatorialist, but a few examples are known.
The following (and others: Stiefel manifolds, constructions using the Hopf map) are discussed in more detail in Matou\v{s}ek's book~\cite{matousek2008using} and in Chapter~3 of Csorba's thesis~\cite{csorbaThesis} devoted to the topic.
The simplest is perhaps the torus with two holes (that is, the 2-dimensional orientable surface of genus 2, with $\ZZ_2$-action $x \mapsto -x$ in a symmetric embedding in $\RR^3$, i.e., swapping the holes)
which has coindex 1 and index 2, that is, $\Sphere^2 \not\to_{\ZZ_2} X$, but $X \not\to_{\ZZ_2} \Sphere^1$.
Real projective spaces  (with an appropriate $\ZZ_2$-action) provide examples with the worst possible gap: they have coindex 1 and arbitrarily high index, that is, $\Sphere^2 \not\to_{\ZZ_2} X$, but $X \not\to_{\ZZ_2} \Sphere^n$, for an arbitrarily high $n$ (the index has been computed by Stolz~\cite{stolz1989level}, see also an exposition in~\cite{pfister1995quadratic}).
Matsushita~\cite{MatsushitaTidy} proved an even stronger example where not only the index is arbitrarily high, but so is a cohomological lower bound of it; his proof also uses considerably fewer tools of algebraic topology.

The  dual to Conjecture~\ref{conj:main}, namely $\coind(X\times Y) = \min(\coind(X), \coind(Y))$, has been considered by Simonyi and Zsb{\'{a}}n~\cite{SimonyiZ10}.
This statement is trivial in topology, that is, $\Sphere^n \to_{\ZZ_2} X \times Y$ if and only if $\Sphere^n \to_{\ZZ_2} X$ and  $\Sphere^n \to_{\ZZ_2} Y$.
However, they showed that  $\coind \BBx{G \times H} = \min(\coind \BBx{G} , \coind \BBx{H})$ (without resorting to $\BBx{G} \times \BBx{H} \zeq \BBx{G \times H}$), which implies that Hedetniemi's conjecture is true on all graphs for which the topological bound on the chromatic number $\coind(\BBx{G}) + 2 \leq \chi(G)$ is tight.
Conjecture~\ref{conj:main} would imply that tightness of the bound $\ind(\BBx{G}) + 2$ would suffice.

In topological literature on the index (see e.g. \cite{Yang54,Yang55,ConnerF60,ConnerF62,Ucci72,Tanaka03}), $\ZZ_2$-maps are usually called equivariant maps. The names `coindex' and `index' are usually swapped with respect to their usage in (topological) combinatorics.
The index has also been called the B-index, level, genus.
Nevertheless, the only mention of the index of products of spaces seems to be~\cite{kaur2013uniform}.
\looseness=-1

We note that the index has important applications in algebra, see~\cite{DaiLP80,DaiL84}; Dai and Lam proved a crucial connection and stated a question~\cite[(11.2)]{DaiL84} about tensor products of commutative $\RR$-algebras that is closely related, via this connection, to Conjecture~\ref{conj:main}. The \emph{level} $s(A)$ of an algebra $A$ is the least $n$ such that $-1$ can be represented as the sum of $n$ squares: $-1=a_1^2+\dots+a_n^2$ for some $a_i \in A$. The question is whether $s(A \otimes_{\RR} B) = \min(s(A), s(B))$, for all commutative $\RR$-affine algebras $A,B$.
As far as we know, this question has not been explored further.\looseness=-1

\subsection*{Open questions}
For a topologist, the main question stemming from this work is of course Conjecture~\ref{conj:main}.
Even though we state it as a conjecture, we have no serious reason to believe it to be true.
In fact so little is known that any partial result would be interesting.
In particular, is $\Sphere^2$, or really any non-1-dimensional $\ZZ_2$-space, multiplicative?
An example of a $\ZZ_2$-space that is \emph{not} multiplicative is $X\times Y$ for any $X,Y$ such that $X \not\to_{\ZZ_2} Y$ and $Y \not\to_{\ZZ_2} X$; can any other examples be given?
As far as we know, it could even turn out that known non-tidy spaces provide relatively simple counter-examples to Conjecture~\ref{conj:main} and hence to Hedetniemi's conjecture.
Can one compute the index of some non-trivial products involving non-tidy spaces?
How about some subspaces of the space of maps from $\Sphere^2$ to $\Sphere^2$?

Closer to combinatorics, we ask how close can the connection with topology be.
Is every graph $K$ with $\BBx{K} \zeq \Sphere^1$ multiplicative?
All known examples suggest so, but very little is known on graphs that are not multiplicative, so any new method for disproving multiplicativity would be interesting.
Beside taking $K=G\times H$ for graphs such that $G \not\to H$ and $H \not\to G$, the only construction known to the author comes from Kneser graphs, see~\cite{TardifZ02a}.

Finally, do other functors have similar properties to $\PathRF{k}$, in particular do all ``adjoint fractional powers'' of the form $\PathCF{\ell}(\PathRF{k}(\cdot))$ with $l < k$ preserve the homotopy type (as in the Equivalence Theorem~\ref{thm:equiv})? How about right adjoints to the arc graph construction?
Can the properties be derived from more general principles?

%% file: smallproofs.tex
We show how the Theorem~\ref{thm:multImplies}, restated here, easily follows from the Equivalence, Approximation and Universality Theorems.
(Recall that a $\ZZ_2$-space $Z$ is \emph{multiplicative} if $X \times Y \to_{\ZZ_2} Z$ implies $X\to_{\ZZ_2} Z$ or $Y\to_{\ZZ_2} Z$ for any $\ZZ_2$-spaces $X,Y$.)

\def\rrrr{-4pt}
\begin{theorem}
Let $K$ be a multiplicative graph. Then $\BBx{K}$ is multiplicative.
\end{theorem}
\begin{proof}
Let $X,Y$, be finite simplicial $\ZZ_2$-spaces such that $X \times Y \to_{\ZZ_2} \BBx{K}$.
By the Universality Theorem~\ref{thm:universality}, there are graphs $G, H$ such that $X \zeq \BBx{G}$ and $Y \zeq \BBx{H}$.
Thus:\vspace*{\rrrr}
	$$\BBx{G \times H} \stackrel{\text{Lem~\ref{lem:boxproduct}}}{\zeq}
	\BBx{G} \times \BBx{H} \zeq X \times Y \to_{\ZZ_2} \BBx{K}.\vspace*{\rrrr}$$
By the Approximation Theorem~\ref{thm:approximation}, there is an odd integer $k$ such that
$\PathRF{k}(G) \times \PathRF{k}(H) \stackrel{\text{Lem~\ref{lem:funcproduct}}}{\leftrightarrow} \PathRF{k}(G \times H) \to K$. 
By definition of multiplicativity of $K$, we have $\PathRF{k}(G) \to K$ or $\PathRF{k}(H) \to K$, hence by the other direction of the Approximation Theorem~\ref{thm:approximation},
$X \zeq \BBx{G} \to_{\ZZ_2} \BBx{K}$ or $Y \zeq \BBx{H} \to_{\ZZ_2} \BBx{K}$.
\end{proof}

The proof of Theorem~\ref{thm:characterization}, restated next, is similarly straightforward.
\begin{theorem}
	Let $Z$ be a $\ZZ_2$-space and let $K$ be a graph such that $Z \zeq \BBx{K}$.
	Then $Z$ is multiplicative if and only if for all graphs $G,H$ the following holds: if $G \times H \to K$, then for some odd $k$, $\PathRF{k}(G) \to K$ or $\PathRF{k}(H) \to K$.
\end{theorem}
\begin{proof}
For one direction, let $K$ be a graph such that $\BBx{K}$ is a multiplicative $\ZZ_2$-space.
Let $G,H$ be graphs and suppose that $G \times H \to K$.
Then\vspace*{\rrrr}
$$\BBx{G} \times \BBx{H} \stackrel{\text{Lem~\ref{lem:boxproduct}}}{\zeq} \BBx{G \times H} \to_{\ZZ_2} \BBx{K}\vspace*{\rrrr}$$
By the multiplicativity of $\BBx{K}$, we have $\BBx{G} \to_{\ZZ_2} \BBx{K}$, which
by the Approximation Theorem~\ref{thm:approximation} implies that $\PathRF{k}(G) \to K$ for some odd integer $k$ (or the same for $H$).

\pagebreak

For the other direction, suppose $K$ has the property that for all graphs $G,H$, $G\times H \to K$ implies  $\PathRF{k}(G) \to K$ or $\PathRF{k}(H) \to K$ for some odd $k$. 
To show that $\BBx{K}$ is a multiplicative space,
let $X,Y$ be any $\ZZ_2$-spaces such that $X \times Y \to_{\ZZ_2} \BBx{K}$.
By the Universality Theorem~\ref{thm:universality}, there are graphs $G,H$ such that $\BBx{G} \zeq X$ and $\BBx{H} \zeq Y$.
Then\vspace*{\rrrr}
$$\BBx{G \times H}  \stackrel{\text{Lem~\ref{lem:boxproduct}}}{\zeq} \BBx{G} \times \BBx{H} \to_{\ZZ_2} \BBx{K}\vspace*{\rrrr}$$
Hence by the Approximation Theorem~\ref{thm:approximation}, there is an odd integer $k$ such that
\vspace*{\rrrr}
$$\PathRF{k}(G) \times \PathRF{k}(H) \stackrel{\text{Lem~\ref{lem:funcproduct}}}{\leftrightarrow} \PathRF{k}(G \times H) \to K\vspace*{\rrrr}$$
By the property of $K$, there is an odd integer $k'$ such that $\PathRF{k'}(\PathRF{k}(G)) \to K$ (or the same for $H$).
By the Approximation Theorem~\ref{thm:approximation} and the Equivalence Theorem~\ref{thm:equiv}, this implies 
$X \zeq \BBx{G} \zeq \BBx{\PathRF{k\cdot k'}(G)} \to_{\ZZ_2} \BBx{K} \zeq Z$ (or $Y \to_{\ZZ_2} Z$).
\end{proof}

%% file: equiv.tex
Throughout this section we assume that $G$ is a graph without loops (hence so is $\PathRF{2k+1}(G)$).
The goal of this section is to show Theorem~\ref{thm:equiv}, in particular that $\BBx{G}$ and $\BBx{\PathRF{2k+1}(G)}$ are $\ZZ_2$-homotopy equivalent, for all $k$.
Following ideas of Csorba~\cite{Csorba08}, we use basics of Discrete Morse Theory, a framework introduced by Forman~\cite{Forman98} which allows to show homotopy equivalence in a very combinatorial way.
We refer to~\cite{Forman02} for an introduction and~\cite{Kozlov2008morse} for an in depth coverage.

Let us introduce the required notions.
We\marginpar{$\sigma \dmatch \tau$} denote the symmetric difference of two simplices as $\sigma \dmatch \tau$.
We will construct homotopy equivalences by composing a sequence of small steps.
If $K$ is a simplicial\marginpar{\footnotesize{collapse}} complex with a simplex $\tau$ such that there is a unique simplex $\sigma\neq \tau$ in $K$ containing $\tau$, then it is not hard to show that $|K \setminus \{\tau,\sigma\}|$ is homotopy equivalent to $|K|$ (and the inclusion map gives one side of a homotopy equivalence); this is called an \emph{elementary collapse}.
If $K'$ can be obtained from $K$ by a sequence of elementary collapses, we say that $K$ \emph{collapses} to $K'$.
If $K'$ can\marginpar{\footnotesize{simple\\\hspace*{0.2em}homot. eq.}} be obtained from $K$ by a sequence of elementary collapses and expansions (operations inverse to elementary collapses), we say that $K'$ is \emph{simple homotopy equivalent} to $K$ (Whitehead showed that this notion is slightly stronger than just homotopy equivalence, see~\cite{Cohen73book}).
The definitions are naturally extended to free $\ZZ_2$-simplicial complexes (where elementary collapses have to be performed in pair: $\tau,\sigma$ are removed together with their $\ZZ_2$-image $-\tau,-\sigma$).

A sequence of elementary collapses can be described more concisely using matchings.
For\marginpar{\vspace{1pt}\hspace*{-3pt}\footnotesize{matching}} a simplicial complex $K$ and a subcomplex $K'$, a \emph{matching} is a bijective function $\mu$ on the set of simplices $K\setminus K'$ such that $\mu \circ \mu = \mathrm{id}$, for each $\sigma\in K\setminus K'$, $\mu(\sigma)$ contains or is contained in $\sigma$, and $|\sigma \dmatch \mu(\sigma)| = 1$.
It is a \emph{$\ZZ_2$-matching} if $\mu(-\sigma)=-\mu(\sigma)$ (the $\ZZ_2$-action applied element-wise).
Since all of the simplices of $K \setminus K'$ are matched into pairs, we can try to order them into a sequence of elementary collapses. The necessary and sufficient condition turns out to be the following.
A matching\marginpar{\footnotesize{acyclic}} is \emph{acyclic} if there is no sequence of containments of the following form (for $n\geq 2$ pair-wise different $\sigma_i$ in $K \setminus K'$):
\begingroup
\newcommand{\rsub}[1]{\rotatebox[origin=c]{#1}{$\subseteq$}}
\newcommand{\rsup}[1]{\rotatebox[origin=c]{#1}{$\supseteq$}}
\setlength{\arraycolsep}{1pt}%
\renewcommand{\arraystretch}{-2}%
\setcounter{MaxMatrixCols}{30}
\vspace*{-10pt}
\begin{equation}\label{eq:cycle}\tag{$\ast$}
\begin{matrix}
&&\mu(\sigma_1)&&&&\mu(\sigma_2)&&&&\mu(\sigma_n)&&\\
&\rsub{45}&&\rsup{-45}&&\rsub{45}&&\rsup{-45}&&\rsub{45}&&\rsup{-45}&\\
\sigma_1&&&&\sigma_2&&&&\ \cdots\ &&&&\sigma_1\\
\end{matrix}
\vspace*{-5pt}
\end{equation}
With those definitions, we can state the basic theorem of Discrete Morse Theory (we note this is only the simplest version of the statement, but we will not need anything more; the $\ZZ_2$ variant was stated in \cite[Remark 7]{Csorba08}, we give a proof for completeness):

\begin{theorem}[]\label{thm:morse}
	Let $K$ be a free $\ZZ_2$-simplicial complex and $K'$ a $\ZZ_2$-subcomplex. If there is an acyclic $\ZZ_2$-matching on the set of simplices $K \setminus K'$, then $K$ $\ZZ_2$-collapses to $K'$.
	In particular, the inclusion map $K' \hookrightarrow K$ is a $\ZZ_2$-homotopy equivalence.
\end{theorem}
\begin{proof}
	Note that acyclicity, i.e. the lack of a sequence of the form \eqref{eq:cycle} is equivalent to the lack of a sequence  of following form, obtained by reversal $\sigma'_i := \mu(\sigma_{n-i+1})$:\vspace*{-5pt}
	\[
	\begin{matrix}
	\sigma'_1&&&&\sigma'_2&&&&\ \cdots\ &&&&\sigma'_1\vspace*{-5pt}\\
	&\rsup{-35}&&\rsub{35}&&\rsup{-35}&&\rsub{35}&&\rsup{-35}&&\rsub{35}&\vspace*{-5pt}\\
	&&\mu(\sigma'_1)&&&&\mu(\sigma'_2)&&&&\mu(\sigma'_n)&&\\
	\end{matrix}
	\vspace*{-3pt}\]
	We prove the statement by induction on $|K \setminus K'|$ (the number of matched simplices).
	Let $\sigma_1'$ be a matched simplex of maximum size.
	If $\mu(\sigma_1')$ is not strictly contained in any other simplex, then $K$ can be collapsed to $K \setminus \{\sigma_1',\mu(\sigma_1'),-\sigma_1',-\mu(\sigma_1')\}$.
	Otherwise $\mu(\sigma_1')$ is strictly contained in some $\sigma_2' \neq \sigma_1'$.
	Since $\mu(\sigma_1')$ is matched, so is $\sigma_2'$ (as unmatched simplices form a subcomplex $K'$).
	We consider $\mu(\sigma_2')$ and repeat the argument; since $\mu$ is acyclic, we eventually collapse something.
\end{proof}
\endgroup

\newcommand{\cut}{\phi}
We will show that $\Bx{\PathRF{2k+1}(G)}$ and $\Bx{\PathRF{2k-1}(G)}$ are (simple) homotopy equivalent by defining an intermediate complex that collapses to both.
For $\tup{A} \in V(\PathRF{2k+1}(G))$, let \marginpar{$\cut(\tup{A})$}
$$\cut(\tup{A}) := \left(A_0, \dots, A_{k-1}, \CN(A_{k-1})\right) \in  V(\PathRF{2k+1}(G)).$$
Define\marginpar{$\PathRF{2k+1}'(G)$} the graph $\PathRF{2k+1}'(G)$ by adding the following edges to $\PathRF{2k+1}(G)$:
for each existing edge $\{\tup{A},\tup{B}\}$, add new edges $\{\tup{A},\cut(\tup{B})\}$, $\{\cut(\tup{A}),\tup{B}\}$, and $\{\cut(\tup{A}),\cut(\tup{B})\}$.
Observe that $\cut(\tup{A})$ is adjacent to $\cut(\tup{B})$ in $\PathRF{2k+1}'(G)$ if and only if $(A_0,\dots,A_{k-1})$ and $(B_0,\dots,B_{k-1})$ are adjacent in $\PathRF{2k-1}(G)$. 
In particular the subgraph of $\PathRF{2k+1}'(G)$ induced on the vertices of $\im \cut \times\{\circ,\bullet\}$ is isomorphic to $\PathRF{2k-1}(G)$.
We show that it induces a $\ZZ_2$-homotopy equivalent subcomplex.

\begin{lemma}\label{lem:easyCollapse}
	$\Bx{\PathRF{2k+1}'(G)}$ $\ZZ_2$-collapses to the subcomplex induced by $\im \cut \times\{\circ,\bullet\}$ (isomorphic to $\Bx{\PathRF{2k-1}(G)}$).
\end{lemma}
\vspace*{-2pt}
\begin{proof}
The simplices not in the subcomplex are exactly those that contain $\vL{\tup{A}}$ (or $\vR{\tup{A}}$) for some vertex $\tup{A}$ from outside $\im \cut$.
We define a matching $\mu$ by matching every such simplex $\sigma$ with $\sigma \dmatch \{\vL{\cut(\tup{A})}\}$ (or $\sigma \dmatch \{\vR{\cut(\tup{A})}\}$), where $\tup{A}$ is chosen to be the smallest vertex in $(\vvL{\sigma} \cup \vvR{\sigma})\setminus \im \cut$, according to an arbitrary, fixed ordering on $V(\PathRF{2k+1}'(G))$.
Note that exactly one of $\vL{\tup{A}}$, $\vR{\tup{A}}$ is in $\sigma$, so this is a well-defined $\ZZ_2$-matching.
The fact that $\sigma \dmatch \{\vU{\cut(\tup{A})}\}$ is a simplex of $\Bx{\PathRF{2k+1}'(G)}$ follows from the definition of $\cut$ and $\PathRF{2k+1}'$; it is outside of $\im \phi\times\{\circ,\bullet\}$ because it still contains~$\vU{\tup{A}}$.

To show that the matching is acyclic, suppose $\sigma_1,\dots,\sigma_n$ ($n \geq 2$) forms a cycle as in~\eqref{eq:cycle}.
When going up the matching, from $\sigma_i$ to $\mu(\sigma_i)$, we always add a vertex in $\im \cut \times\{\circ,\bullet\}$.
Therefore, since the sequence forms a cycle, when going down from $\mu(\sigma_i)$ to $\sigma_{i+1}$ we can only remove vertices in $\im \cut \times\{\circ,\bullet\}$;
the set of vertices in $\sigma_i \setminus (\im \cut \times\{\circ,\bullet\})$ remains constant. 
But then the vertex $\vU{\cut(\tup{A})}$ added when going up the matching from $\sigma_1$ to $\mu(\sigma_1)$ is also the vertex in  $\sigma_2 \dmatch \mu(\sigma_2)$, by definition of the matching $\mu$.
This vertex is not removed when going down from $\mu(\sigma_1)$ to $\sigma_2$, since $\sigma_1 \neq \sigma_2$ ($n \geq 2$).
Hence $\sigma_2$ contains this vertex and $\mu(\sigma_2) = \sigma_2 \setminus \{\vU{\cut(\tup{A})}\}$,
contradicting the fact that the sequence should go up the matching ($\sigma_2 \subseteq \mu(\sigma_2)$).%
\vspace*{-2pt}%
\end{proof}

\noindent
To show the collapse to $\Bx{\PathRF{2k+1}(G)}$, let us describe minimal simplices that have to be collapsed.

\begin{lemma}\label{lem:preCollapse}
	A simplex $\sigma \in \Bx{\PathRF{2k+1}'(G)}$ is not in the subcomplex $\Bx{\PathRF{2k+1}(G)}$ if and only if\looseness=-1
	\begin{enumerate}[label=(\roman*)]
	\item $\sigma$ contains $\vL{\tup{A}},\vR{\tup{B}}$ such that $A_k \notjoin B_k$, or
	\item $\sigma$ contains $\vL{\tup{A}},\vL{\tup{C}}$  (or  $\vR{\tup{A}},\vR{\tup{C}}$) such that $A_k \notjoin C_{k-1}$.
	\end{enumerate}
\end{lemma}
\begin{proof}
Let us first show one direction.
If $\sigma \in \Bx{\PathRF{2k+1}'(G)}$ contains $\vL{\tup{A}},\vR{\tup{B}}$ such that $A_k \notjoin B_k$, then these are clearly not adjacent in $\PathRF{2k+1}(G)$, so $\sigma$ is not in the subcomplex.
If $\sigma$ contains $\vL{\tup{A}},\vL{\tup{C}}$  (or  $\vR{\tup{A}},\vR{\tup{C}}$) such that $A_k \notjoin C_{k-1}$, then suppose $\sigma$ is in the subcomplex.
By definition this implies that $\CN(\vvL{\sigma})$ (meaning the common neighborhood in $\PathRF{2k+1}(G)$) is non-empty, so let~$\tup{B}$~be~a~common neighbor of $\tup{A}$ and $\tup{C}$ in $\PathRF{2k+1}(G)$.
Then $\tup{A}$ is adjacent to $\tup{B}$ in $\PathRF{2k+1}(G)$, which implies $A_k \join B_k$, and $\tup{B}$ is adjacent to $\tup{C}$, which implies $B_k \supseteq C_{k-1}$, contradicting $A_k \notjoin C_{k-1}$. Hence $\sigma$ cannot be in the subcomplex.

For the other direction, consider a simplex $\sigma \in \Bx{\PathRF{2k+1}'(G)}$ that is not in the subcomplex.
That is, there are $\vL{\tup{A}},\vR{\tup{B}} \in \sigma$ such that $\tup{A}$ and $\tup{B}$ are not adjacent in $\PathRF{2k+1}(G)$, or it must be that $\CN(\vvL{\sigma})$ or $\CN(\vvR{\sigma})$ is empty.
In the former case, since $\tup{A}$ and $\tup{B}$ are adjacent in $\PathRF{2k+1}'(G)$, we conclude that $A_k \notjoin B_k$.
In the latter case,  say $\CN(\vvL{\sigma})$ is empty. That is, the vertices of $\vvL{\sigma}$ do not have a common neighbor in $\PathRF{2k+1}(G)$, although they do have some common neighbor $\tup{B}$~in~$\PathRF{2k+1}'(G)$.
Let $\tup{B}' := (B_0,\dots,B_{k-1}, B_k')$ where $B_k' := \bigcup_{\tup{C} \in \vvL{\sigma}} C_{k-1}$ (this is a vertex of $\Omega_{2k+1}(G)$, because $C_{k-1} \join B_{k-1}$ for $\tup{C} \in \vvL{\sigma}$, because $\tup{B}$ is a neighbor in $\PathRF{2k+1}'(G)$).
Since $\tup{B}'$ in particular is not a common neighbor of $\vvL{\sigma}$ in $\PathRF{2k+1}(G)$, it must be that $B_k' \notjoin A_k$ for some $\tup{A} \in \vvL{\sigma}$. By definition of $B_k'$, this means that $C_{k-1} \notjoin A_k$ for some $\tup{A},\tup{C} \in \vvL{\sigma}$.
\end{proof}

We can now show the necessary collapse, in phases corresponding to the points in Lemma~\ref{lem:preCollapse}.
The reader is warned that the proof is not very illuminating, it is just trying the simplest collapses that come to mind, carefully adapted into a few phases until all cases are covered, and checking that all the technical conditions are satisfied.

\begin{lemma}\label{lem:hardCollapse}
	$\Bx{\PathRF{2k+1}'(G)}$ $\ZZ_2$-collapses to the subcomplex $\Bx{\PathRF{2k+1}(G)}$.
\end{lemma}
\begin{proof}
We first collapse simplices containing some vertices $\vL{\tup{A}},\vL{\tup{C}}$ (or  $\vR{\tup{A}},\vR{\tup{C}}$) such that $A_k \notjoin C_{k-1}$.
Among those, we first collapse simplices where $\tup{A}$ can be chosen from outside $\im \cut$.

For any such simplex $\sigma$, choose $\vU{\tup{A}} \in \sigma \setminus \im \cut$, $\vU{\tup{C}} \in \sigma$ such that $A_k \notjoin C_{k-1}$ and $(\tup{A},\tup{C})$ is lexicographically minimum (according to some arbitrary fixed ordering of vertices of $\PathRF{2k+1}(G)$). Without loss of generality assume $\unk = \circ$ for this minimum pair.
Let\vspace*{-5pt}
 $$\tup{A}^* := (A_0,\dots,A_{k-1},\CN(S))  \quad\text{ where }\quad  S := \bigcup_{\tup{A}'\in\vvL{\sigma}} A'_{k-1} \quad \cup \quad \bigcup_{\tup{B}' \in \vvR{\sigma} \setminus \im\cut} B'_k\vspace*{-5pt}$$
We define a matching $\mu(\sigma) := \sigma \dmatch \{\vL{\tup{A}^*}\}$.
We need to check a series of technical conditions:
\begin{enumerate}[label=(\roman*)]
	\item\label{lI-1} the vertex $\tup{A}^*$ is well-defined;
	\item\label{lI-2} $\sigma \dmatch \{\vL{\tup{A}^*}\}$ is a simplex of $\Bx{\PathRF{2k+1}'(G)}$; equivalently, that $\tup{A}^*$ is adjacent to vertices in $\vvR{\sigma}$ and has a common neighbor together with all the vertices in $\vvL{\sigma}$;
	\item\label{lI-3} $\sigma \dmatch \{\vL{\tup{A}^*}\}$ still contains $\vL{\tup{A}}$ and $\vL{\tup{C}}$ (so it is not in the subcomplex we collapse to);
	\item\label{lI-4} $\mu(\sigma \dmatch \{\vL{\tup{A}^*}\}) = \sigma$ (so that $\mu$ is indeed a matching);
	\item\label{lI-5} $\mu$ is acyclic.
\end{enumerate}
For \ref{lI-1}, observe that $S$ contains $A_{k-1}$, which implies $A_{k-1} \join \CN(S)$, as required for a vertex.
\pagebreak[3]

For \ref{lI-2}, let us first show that $\tup{A}^*$ is adjacent to each vertex in $\vvR{\sigma}$. 
Let $\tup{B} \in \vvR{\sigma}$.
One condition for adjacency is that $B_{k-1} \subseteq \CN(S)$, or equivalently, that $B_{k-1} \join S$.
This holds, because $B_{k-1} \join A'_{k-1}$ for each $\tup{A}' \in \vvL{\sigma}$ (because $\vR{\tup{B}},\vL{\tup{A}'}$ are adjacent, as they are contained in $\sigma$).
Furthermore, $B_{k-1} \join B'_k$ for each $\tup{B}' \in \vvR{\sigma}\setminus \im \cut$, because $B_{k-1} \subseteq A_k$ and $A_k \join B'_k$ (because $\vR{\tup{B}}$ and $\vL{\tup{A}}$ are adjacent, while $\vL{\tup{A}}$ and $\vR{\tup{B}'}$ are adjacent and not in $\im \cut$).
Thus $B_{k-1} \join S$, that is, $B_{k-1} \subseteq \CN(S)$.

If $\tup{B} \in \im \cut$, then $\tup{B} = (B_0,\dots,B_{k-1},\CN(B_{k-1})) = \cut((B_0,\dots,B_{k-1},A_{k-1}))$. But $\tup{A}^*$ is adjacent to $(B_0,\dots,B_{k-1},A_{k-1})$, because $\tup{A}$ was adjacent to $\tup{B}$, $B_{k-1} \subseteq \CN(S)$ and $A_{k-1} \join \CN(S)$.
Hence  $\tup{A}^*$ is (by definition of $\PathRF{2k+1}'$) also adjacent to $\cut((B_0,\dots,B_{k-1},A_{k-1}))$, which is $\tup{B}$.

If on the other hand $\tup{B} \not\in \im \cut$, then $\tup{A}^*$ is again adjacent to it, because $\tup{A}$ was, $B_{k-1} \subseteq \CN(S)$ (as shown above), and $B_k \join \CN(S)$ (because $B_k \subseteq S$).

To conclude \ref{lI-2}, it remains to show that $\tup{A}^*$ has a common neighbor together with all vertices in $\vvL{\sigma}$.
If $\vvR{\sigma}$ is non-empty, then any vertex in it is such a common neighbor.
If however $\vvR{\sigma}$ is empty, then there must exists a vertex $\tup{B} \in \CN(\vvL{\sigma})$, so $\sigma \cup \{\vR{\tup{B}}\}$ is a simplex and the same analysis as above shows that  $\tup{A}^*$ is also adjacent to $\tup{B}$, proving that $\tup{B}$ is a common neighbor of~$\vvL{\sigma} \dmatch \{\tup{A}^*\}$.

For \ref{lI-3}, we need to show that $\vL{\tup{A}^*} \neq \vL{\tup{A}}$ and $\vL{\tup{A}^*} \neq \vL{\tup{C}}$.
The former follows from the fact that $S \supseteq C_{k-1}$, so $A^*_{k} = \CN(S) \join C_{k-1}$, while $A_k \notjoin C_{k-1}$, thus $A^*_k \neq A_k$.
The latter follows from the fact that $C_{k-1} \notjoin A_k$, but $A^*_{k-1} = A_{k-1} \join A_k$, so $C_{k-1} \neq A^*_{k-1}$.

For \ref{lI-4}, we need to show that the initial choice of a pair $\vU{\tup{A}},\vU{\tup{C}}$ for $\sigma \cup \{\vL{\tup{A}^*}\}$ will be the same as for $\sigma \setminus \{\vL{\tup{A}^*}\}$.
Recall that valid choices are pairs $\vU{\tup{A}},\vU{\tup{C}}$ of vertices in the simplex such that $\tup{A} \not\in \im\cut$ and $A_{k} \notjoin C_{k-1}$, and we select the lexicographically minimum valid choice.
Without loss of generality assume $\sigma \dmatch \{\vL{\tup{A}^*}\} = \sigma \cup \{\vL{\tup{A}^*}\}$ and
suppose to the contrary that the choice for $\sigma \cup \{\vL{\tup{A}^*}\}$ is $(\tup{A}^\dagger,\tup{C}^\dagger)$, different from the choice $(\tup{A},\tup{C})$ for $\sigma$.
Since $(\tup{A},\tup{C})$ is a valid choice for $\sigma \cup \{\vL{\tup{A}^*}\}$ as well, $(\tup{A}^\dagger,\tup{C}^\dagger)$ must be lexicographically smaller.
That is, either $\tup{A}^\dagger <  \tup{A}$, or $\tup{A}^\dagger = \tup{A}$ and $\tup{C}^\dagger < \tup{C}$.
Since $(\tup{A}^\dagger,\tup{C}^\dagger)$ was not a valid choice for $\sigma$, we have $\tup{A}^\dagger = \tup{A}^*$ or $\tup{C}^\dagger = \tup{A}^*$.
If $\tup{A}^\dagger = \tup{A}^*$, then $A^\dagger_{k} = A^*_k = \CN(S) \join C^\dagger_{k-1}$ (because $S \supseteq C^\dagger_{k-1})$, contradicting the fact that $(\tup{A}^\dagger, \tup{C}^\dagger)$ was a valid choice.
If on the other hand $\tup{C}^\dagger = \tup{A}^*$, then the validity of the choice implies $A^\dagger_k \notjoin C^\dagger_{k-1} = A^*_{k-1} = A_{k-1}$. In particular $A^\dagger \neq A$, so $A^\dag < A$. But then the pair $(A^\dag, A)$ would have been a valid, lexicographically smaller choice for $\sigma$, a contradiction.

Finally we show \ref{lI-5}, that is, the matching $\mu$ is acyclic.
Suppose to the contrary that $\sigma_1,\dots,\sigma_n$ ($n \geq 2$) forms a cycle as in~\eqref{eq:cycle}.
When going up the matching, from $\sigma_i$ to $\mu(\sigma_i)$, the initial choice of $\vU{\tup{A}},\vU{\tup{C}}$ remains the same, as shown in \ref{lI-4}.
When going down from $\mu(\sigma_i)$ to  $\sigma_{i+1}$ contained in it, the initial choice can only stay the same or increase lexicographically (since it is also available for $\mu(\sigma_i)$).
Hence the choice of $\vU{\tup{A}},\vU{\tup{C}}$ must in fact remain unchanged throughout the cycle, say it is $\vL{\tup{A}},\vL{\tup{C}}$ for all $\sigma_i$ and $\mu(\sigma_i)$.
Therefore, the vertices we add (and hence also those we remove) in the cycle are all of the form $\vL{(A_0,\dots,A_{k-1},X)}$ for some vertex subsets $X$.
This implies that the set $S$, as defined above, and hence also $\tup{A}^*$, is always the same when defining the simplex $\mu(\sigma_i)$ matched to $\sigma_i$.
But then $\tup{A}^*$ is always the vertex added (the vertex in $\mu(\sigma_i) \setminus \sigma_i$) and hence also the only vertex removed (the one in $\mu(\sigma_i) \setminus \sigma_{i+1}$), which implies $\sigma_1 = \sigma_2$, a contradiction.

\bigskip
We now collapse the remaining simplices $\sigma$ that contain some vertices $\vL{\tup{A}},\vL{\tup{C}}$ (or  $\vR{\tup{A}},\vR{\tup{C}}$) such that $A_k \notjoin C_{k-1}$.
By the previous collapsing phase, we know that $\tup{A} \in \im\cut$ and symmetrically:
\begin{equation}\label{lII}
\text{For any  $\vR{\tup{B}'},\vR{\tup{B}} \in \sigma $ with $B'_k \notjoin B_{k-1}$, we know that $\tup{B'} \in \im\cut$.}
\end{equation}
For any such simplex $\sigma$, choose $\vU{\tup{A}}, \vU{\tup{C}} \in \sigma$ such that $A_k \notjoin C_{k-1}$ and $(\tup{A},\tup{C})$ is lexicographically minimum (according to some arbitrary fixed ordering of vertices of $\PathRF{2k+1}(G)$).
Without loss of generality assume $\unk = \circ$ for this minimum pair.
Just as before, let $$\textstyle \tup{A}^* := (A_0,\dots,A_{k-1},\CN(S)) \quad\text{ where }\quad S := \bigcup_{\tup{A}'\in\vvL{\sigma}} A'_{k-1} \quad \cup \quad \bigcup_{\tup{B}' \in \vvR{\sigma} \setminus \im\cut} B'_k$$
We define a matching $\mu(\sigma) := \sigma \dmatch \{\vL{\tup{A}^*}\}$.
Similarly as before, we need to show \ref{lI-1}--\ref{lI-5}.
The proof of \ref{lI-1} is unchanged: $S$ contains $A_{k-1}$, which implies $A_{k-1} \join \CN(S)$, as required for a vertex.

For \ref{lI-2}, let us first show that $\tup{A}^*$ is adjacent to each vertex in $\vvR{\sigma}$. 
Let $\tup{B} \in \vvR{\sigma}$.
Observe that $B_{k-1} \join A'_{k-1}$ for all $\tup{A}' \in \vvL{\sigma}$ and by $\eqref{lII}$, $B_{k-1} \join B'_k$ for $\tup{B}' \in \vvR{\sigma} \setminus \im \cut$, hence $B_{k-1} \join S$, which means $B_{k-1} \subseteq \CN(S)$.
The remaining proof proceeds just as before (with two cases depending on  $\tup{B} \in \im\cut$ or  $\tup{B} \not\in \im\cut$), concluding \ref{lI-2}.
The proofs of \ref{lI-3}--\ref{lI-5} also proceed without change, since they never used the fact that $\tup{A} \not\in \im\cut$.

\bigskip
Finally, we collapse all simplices $\sigma$ containing $\vL{\tup{A}},\vR{\tup{B}}$ such that $A_k \notjoin B_k$.
Fortunately this is considerably simpler, since $\vvR{\sigma}$ is non-empty, and by the previous collapses, we know that 
\begin{equation}\label{lIII}
	\text{$B'_{k} \join B''_{k-1}$ for any $\tup{B}',\tup{B}'' \in \vvR{\sigma}$}
\end{equation}

For any such $\sigma$, choose a lexicographically minimum pair $(\vL{\tup{A}},\vR{\tup{B}})$ or $(\vR{\tup{A}},\vL{\tup{B}})$  in $\sigma$ such that $A_k \notjoin B_k$.
Without loss of generality assume it is $(\vL{\tup{A}},\vR{\tup{B}})$.
Let $\tup{A}^* := (A_0,\dots,A_{k-1},\bigcup_{\tup{B}' \in \vvR{\sigma}} B'_{k-1})$.
We define a matching $\mu(\sigma) := \sigma \dmatch \{\vL{\tup{A}^*}\}$ and check \ref{lI-1}--\ref{lI-5}.
It is now easy to check (using $\eqref{lIII}$) that \ref{lI-1} and \ref{lI-2} are satisfied.

For \ref{lI-3}, we need to show that $\vL{\tup{A}^*} \neq \vL{\tup{A}}$ (and trivially $\vL{\tup{A}^*} \neq \vR{\tup{B}}$).
This follows from the fact that $A_k \notjoin B_k$, but $A^*_k = \bigcup_{\tup{B}' \in \vvR{\sigma}} B'_{k-1} \join B_k$ (by $\eqref{lIII}$).

For \ref{lI-4}, we need to show that the initial choice of a pair $\vL{\tup{A}},\vR{\tup{B}}$ for $\sigma \cup \{\vL{\tup{A}^*}\}$ will be the same as for $\sigma \setminus \{\vL{\tup{A}^*}\}$.
This follows from the fact that $A^*_k \join B'_k$ for all $\tup{B}' \in \vvR{\sigma}$, so $\tup{A}^*$ does not contribute in any way to this choice (since it only considers vertices such that $A_k \notjoin B_k$).

Finally to show \ref{lI-5}, suppose to the contrary that $\sigma_1,\dots,\sigma_n$ ($n \geq 2$) forms a cycle as in~\eqref{eq:cycle}.
When going up the matching, from $\sigma_i$ to $\mu(\sigma_i)$, the initial choice of $\vL{\tup{A}},\vR{\tup{B}}$ remains the same, as shown in \ref{lI-4}.
When going down from $\mu(\sigma_i)$ to the simplex $\sigma_{i+1}$ contained in it, the initial choice can only stay the same or increase lexicographically (since it is also available for $\mu(\sigma_i)$).
Hence the choice of $\vL{\tup{A}},\vR{\tup{B}}$ must in fact remain unchanged throughout the cycle.
This implies that when going up the matching, we only add vertices to $\vvL{\sigma}$, so when going through the cycle we also only remove vertices from $\vvL{\sigma}$, and $\vvR{\sigma}$ is unchanged.
But then the vertex $\tup{A}^*$ added in the matching (in $\mu(\sigma_i) \setminus \sigma_i$) is always the same, so the only possible vertex in $\mu(\sigma_i) \setminus \sigma_{i+1}$ is also $\tup{A}^*$, implying that $\sigma_1=\sigma_2$, a contradiction.
\end{proof}

Theorem~\ref{thm:morse} with Lemma~\ref{lem:easyCollapse} and~\ref{lem:hardCollapse} already imply the $\ZZ_2$-simple homotopy equivalence of $\Bx{\PathRF{2k+1}(G)}$ and  $\Bx{\PathRF{2k-1}(G)}$.
To describe an explicit homotopy equivalence, Theorem~\ref{thm:morse} is insufficient, as it only guarantees a map in one direction of a collapse (the containment map) to be a homotopy equivalence. We hence replace the use of Lemma~\ref{lem:easyCollapse} with an explicit homotopy to conclude our theorem.

\begin{theorem}\label{thm:equivStep}
 For a graph $G$ without loops and $k \in \mathbb{N}$, $\Bx{\PathRF{2k+1}(G)}$ and $\Bx{\PathRF{2k-1}(G)}$ are 
$\ZZ_2$-simple homotopy equivalent.
 Moreover, the homomorphism $\PathRF{2k+1}(G) \to \PathRF{2k-1}(G)$ given by $(A_0,\dots,A_{k-1},A_k) \mapsto (A_0,\dots,A_{k-1})$ induces a $\ZZ_2$-homotopy equivalence.
\end{theorem}
\begin{proof}
	By Theorem~\ref{thm:morse} and Lemma~\ref{lem:hardCollapse} the containment map of $\Bx{\PathRF{2k+1}(G)}$ in $\Bx{\PathRF{2k+1}'(G)}$ is a $\ZZ_2$-homotopy equivalence.
	It remains to show that the following map $q$ from $\Bx{\PathRF{2k+1}'(G)}$ to $\Bx{\PathRF{2k-1}(G)}$
	is a $\ZZ_2$-homotopy equivalence:
	$$q\colon \vU{(A_0,\dots,A_{k-1},A_k)} \mapsto \vU{(A_0,\dots,A_{k-1})} \quad (\text{for }\unk \in \{\circ,\bullet\}) $$
	Consider the containment map $\iota: \vU{(A_0,\dots,A_{k-1})} \mapsto \vU{(A_0,\dots,A_{k-1},\CN(A_k))}$ of $\Bx{\PathRF{2k-1}(G)}$ in $\Bx{\PathRF{2k+1}'(G)}$.
	One composition, $q \circ \iota: \Bx{\PathRF{2k-1}(G)} \to \Bx{\PathRF{2k-1}(G)}$, is just the identity.
	
	The other composition is $\iota \circ q: \vU{(A_0,\dots,A_{k-1},A_k)} \mapsto \vU{(A_0,\dots,A_{k-1},\CN(A_{k-1}))}$.
	For $t\in [0,1]$, define $q_t(\tup{A}) = (1-t) \cdot \tup{A} + t \cdot \iota(q(\tup{A})) \in \BBx{\PathRF{2k+1}'(G)}$ and extend it linearly from vertices to all simplices.
	To show that this is well-defined, we need to show that for any simplex $\sigma \in \Bx{\PathRF{2k+1}'(G)}$, the set $\sigma \cup \{\iota(q(\vU{\tup{A}})) \mid \vU{\tup{A}} \in \sigma\}$ is again a simplex.
	This follows from the definition of $\PathRF{2k+1}'$ and the fact that $\iota \circ q$ coincides with the map $\cut$ used in this definition.
	Thus $q_t$ defines a $\ZZ_2$-homotopy from $q_1 = \iota \circ q$ to $q_0$, the identity map.
	Therefore $q$ and $\iota$ are $\ZZ_2$-homotopy equivalences, and hence $q$ composed with the containment map of $\Bx{\PathRF{2k+1}(G)}$ into $\Bx{\PathRF{2k+1}'(G)}$ is a $\ZZ_2$-homotopy equivalence.
\end{proof}

We conclude the proof of the Equivalence Theorem~\ref{thm:equiv} by applying Theorem~\ref{thm:equivStep} repeatedly.

%% file: approx.tex
We first show that $\BBx{\PathRF{2k+1}(G)}$ refines $\BBx{G}$.
This is the part where using $\PathRF{2k+1}$ instead of iterations of $\PathRF{3}$ makes the proof considerably simpler, because when iteratively applying $\PathRF{3}$, one would have to deal with the fact that degrees in the graph, and hence dimensions of simplices, grow exponentially.

\pagebreak[3]

\begin{theorem}\label{thm:technical}
 Let $G$ be loop-free.
 There is a $\ZZ_2$-map $g \colon \BBx{\PathRF{2k+1}(G)} \to \BBx{G}$ such that:
 \begin{itemize}
	\item $g$ is $\ZZ_2$-homotopic to the map induced by the homomorphism $p_k : \PathRF{2k+1}(G) \to G$,
	\item in particular, $g$ is a $\ZZ_2$-homotopy equivalence (by Theorem~\ref{thm:equiv}),
	\item $g$ maps every simplex of $\Bx{\PathRF{2k+1}(G)}$ into a subset of $\BBx{G}$ of diameter less than $\frac{6D}{k}$, where $D$ is the maximum degree of $G$.
\end{itemize}
\end{theorem}
\vspace*{-3pt}
\begin{proof}\vspace*{-1pt}\setlength{\abovedisplayskip}{4pt}\setlength{\belowdisplayskip}{4pt}
	For a set of points $S = \{s_1,\dots,s_n\} \subseteq \BBx{G}$, define $\avg(S) := \frac{1}{n}(s_1+\dots+s_n)$.
	For $\vL{\tup{A}} \in V(\Bx{\PathRF{2k+1}(G)})$, let 
	$$g(\vL{\tup{A}}) := \avg(\{\avg(\vL{A_0}),\avg(\vR{A_1}),\avg(\vL{A_2}),\dots,\avg(\vU{A_k})\}$$
	(with $\circ$ and $\bullet$ alternating). 
	Since $\tup{A}$ is not an isolated vertex (by definition of the box complex), we have $A_i\subseteq A_{i+2}$ for $i= 0\dots k-2$ and $A_i \join A_{i+1}$ for $i=0 \dots k-1$.	
	Hence the set $\tau_{\vL{\tup{A}}} := \bigcup_{i\text{ even}} \vL{A_i} \cup \bigcup_{i\text{ odd}} \vR{A_i}$ is a simplex of $\Bx{G}$.
	As $g(\vL{\tup{A}})$ is a convex combination of vertices in this simplex, it defines a point in the geometric realization $\BBx{G}$.
	Define $g(\vR{\tup{A}}) \in \BBx{G}$ symmetrically and extend this map linearly to all of $\BBx{\PathRF{2k+1}(G)}$.
	This is well defined, since the set $\tau_\sigma := \bigcup_{\vU{\tup{A}} \in \sigma} \tau_{\vU{\tup{A}}}$ is easily checked to be a simplex of $\Bx{G}$.
	
	For $t\in [0,1]$ and $\vL{\tup{A}} \in V(\Bx{\PathRF{2k+1}(G)})$ with $A_0 = \{v\}$, define $g_t(\vL{\tup{A}}) := (1-t) \cdot g(\vL{\tup{A}}) + t \cdot \vL{v}$, similarly for $\vR{\tup{A}}$.
	This is again a convex combination of vertices in $\tau_{\vL{\tup{A}}}$ that extends linearly, hence $g_t$ defines a $\ZZ_2$-homotopy from $g_0 = g$ to $g_1 = p_k$.

	It remains to bound the diameter of images of simplices.
	Let $\sigma$ be a maximal simplex of $\Bx{\PathRF{2k+1}(G)}$.
	Let $\vL{\tup{A}},\vR{\tup{B}} \in \sigma$.
	Then $A_0 \subseteq B_1 \subseteq A_2 \subseteq \dots$ and $B_0 \subseteq A_1 \subseteq B_2 \subseteq \dots$ are two subset chains of length $k+1$.
	Moreover, since $A_k \join B_k$, all sets have size at most $|A_k|,|B_k| \leq D$, where $D$ is the maximum degree in $G$.
	Since $|A_0| = |B_0| = 1$, in both chains at most $D-1$ containments $\subseteq$ are strict, all the other ones are equalities.
	In particular, among the pairs $(A_0,B_1),(A_1,B_0),(A_2,B_3),(A_3,B_2),\dots$ at most $2D-2$ are pairs of different sets.
	Therefore $g(\vL{\tup{A}})$ and $g(\vR{\tup{B}})$ can be written as $\frac{2D - 2}{k}\,p + \frac{k - (2D-2)}{k}\,q$ and $\frac{2D - 2}{k}\,p' + \frac{k - (2D-2)}{k}\,q$ for some points $p,p',q$ in $\tau_\sigma$ (defined as averages of subsets of $\{\avg(\vL{A_0}),\avg(\vR{A_1}), \dots\}$ and $\{\avg(\vR{B_0}), \avg(\vL{B_1}), \dots\}$).
	Hence the distance between  $g(\vL{\tup{A}})$ and $g(\vR{\tup{B}})$ is the same as the distance between $\frac{2D - 2}{k}\,p$ and  $\frac{2D - 2}{k}\,p'$, which is at most $\frac{2D - 2}{k}\,\sqrt{2}$, since the distance between any two points $p,p'$ in a simplex (namely in $\tau_\sigma$) of a geometric realization is at most $\sqrt{2}$.
	The distance between $g(\vL{\tup{A}})$ and $g(\vL{\tup{A}'})$ for any $\vL{\tup{A}},\vL{\tup{A}'} \in \sigma$ is at most the sum of their distances to $g(\vR{\tup{B}})$, hence at most $\frac{4\sqrt{2}(D-1)}{k}$.
	Since the image of $\sigma$ is the convex hull of the images of its vertices, its diameter is also bounded by $\frac{4\sqrt{2}(D-1)}{k}$, which we bound by $\frac{6D}{k}$ for conciseness.
	\vspace*{-2pt}
\end{proof}

We are now ready to show the Approximation Theorem~\ref{thm:approximation}, restated next.
This closely follows the standard technique of simplicial approximation (see e.g. Theorem~2C.1. in~\cite{Hatcher}). The main difference is that we consider $\ZZ_2$-maps instead of just continuous maps, and that finding a $\ZZ_2$-map between box complexes that is simplicial (i.e., a linear extension of a map on vertices of the complex) is not enough to find a graph homomorphism (because a simplicial map can map two adjacent vertices into a single vertex).
To avoid these problems, we use a variant of the fact that the box complex is $\ZZ_2$-homotopy equivalent to another complex (called $\Hom(K_2,G)$)~\cite{Csorba08}, allowing us to avoid certain extremal points of the box complex.

\begin{theorem}
 There exists a $\ZZ_2$-map from $\BBx{G}$ to $\BBx{H}$ if and only if for some $k\in\NN$, $\PathRF{2k+1}(G)$ has a homomorphism to $H$.
\smallskip

 Moreover, if $G$ and $H$ have no loops, then for any $\ZZ_2$-map $f:\BBx{G} \to_{\ZZ_2} \BBx{H}$ there is an integer~$k$ and a homomorphism $\PathRF{2k+1}(G) \to H$ that induces a map $\ZZ_2$-homotopic to $f \circ |p_k|$.
\end{theorem}
\begin{proof}
	The first statement is trivial if $G$ or $H$ have loops.
	Namely, if $H$ has a loop at $v$, then every graph admits a homomorphism to it and $\BBx{H}$ is not free---the $\ZZ_2$-action maps $\frac{1}{2}\vL{v}+\frac{1}{2}\vR{v}$ to itself---so every $\ZZ_2$-space trivially admits a $\ZZ_2$-map to it.
	If $G$ has a loop, then $\PathRF{2k+1}(G)$ also does, so it only admits homomorphisms to graphs with loops.
	Similarly, $\BBx{G}$ is not free, so it admits $\ZZ_2$-maps only to non-free $\ZZ_2$-spaces.
	We hence assume that $G$ and $H$ have no loops.

	For one direction, suppose $\PathRF{2k+1}(G)$ has a homomorphism to $H$, for some $k\in\NN$.
	This induces a $\ZZ_2$-map from $\BBx{\PathRF{2k+1}(G)}$ to $\BBx{H}$.
	By Theorem~\ref{thm:equiv}, there is a $\ZZ_2$-map from $\BBx{G}$ to $\BBx{\PathRF{2k+1}(G)}$.
	Composition then gives a $\ZZ_2$-map from $\BBx{G}$ to $\BBx{H}$.
	\medskip

	For the other direction, let $f:\BBx{G} \to_{\ZZ_2} \BBx{H}$ be a $\ZZ_2$-map.
	Let $X \subseteq \BBx{H}$ be the set of points $x \in \BBx{H}$ that, when written as a convex combination $x=\sum_{v \in V(H)} \lambda_v \vL{v} + \sum_{v \in V(H)} \mu_v \vR{v}$, satisfy $\sum_v \lambda_v = \sum_v \mu_v = \frac{1}{2}$.
	Simonyi et al.~\cite{SimonyiTV09} observed that the equivalences of various versions of the box complex imply that $\BBx{H}$ is $\ZZ_2$-homotopy equivalent to the subspace $X$.
	Therefore, up to $\ZZ_2$-homotopy, we can assume that the image of $f$ is contained in~$X$.\looseness=-1

	Define the \emph{star} of a simplex $\sigma$ in a simplicial complex $K$ to be the subcomplex made of all simplices containing $\sigma$, and all their subsets, that is: $\{\tau \mid \tau \cup \sigma \in K\}$.
	Define the \emph{closed star} $\St \sigma \subseteq \left|K\right|$ as the geometric realization of the star of $\sigma$ and the \emph{open star} $\st \sigma \subseteq \left|K\right|$ as the sum of interiors of geometric realizations of simplices in the star of $\sigma$.
	Thus $\St \sigma$ is the closure of the open set $\st \sigma$.
	For a vertex $v$ of the complex, we write $\St v$ for short instead of $\St \{v\}$.

	Observe that the sets $\st \vL{v}$ for $v \in V(H)$ cover (a superset of) $X$ in $\BBx{H}$ (to cover all of $\BBx{H}$ we would need $\st \vR{v}$ as well).
	Consider the family of sets $\vL{\Cc} := \{f^{-1}(\st \vL{v}) \mid v \in V(H)\}$.
	This is a family of open sets covering $\BBx{G}$ (because $f$ is continuous into $X$), which is a compact space (as a closed and bounded subset of $\RR^n$).
	Therefore, we can let $\varepsilon > 0$ be the Lebesgue number of $\vL{\Cc}$, that is, a number such that any set of points in $\BBx{G}$ of diameter less than $\varepsilon$ is contained in some set of $\vL{\Cc}$.
	
	Let $D$ be the maximum degree of $G$ and let $k := 12 D \cdot \frac{1}{\varepsilon}$.
	Let $g$ be the $\ZZ_2$-map from $\BBx{\PathRF{2k+1}(G)}$ to $\BBx{G}$ given by Theorem~\ref{thm:technical}.
	For every simplex $\sigma$ of $\BBx{\PathRF{2k+1}(G)}$, $g(\sigma)$ has diameter less than $\frac{6D}{k}$ in $\BBx{G}$.
	Thus, for every vertex $\tup{A}$ of $\PathRF{2k+1}(G)$, $g(\St \vL{\tup{A}})$ has diameter less than 
	$\frac{12 D}{k} = \varepsilon$ in $\BBx{G}$,
	hence there is a vertex $h(\tup{A}) \in V(H)$ such that $g(\St \vL{\tup{A}}) \subseteq f^{-1}(\st\vL{h(\tup{A})})$,
	that is, $f(g(\St \vL{\tup{A}})) \subseteq \st\vL{h(\tup{A})}$.
	We claim that $h: V(\PathRF{2k+1}(G)) \to V(H)$ is a graph homomorphism.

	\medskip
	Indeed, let $\tup{A}, \tup{B}$ be two adjacent vertices of $\PathRF{2k+1}(G)$.
	Then $f(g(\St \vL{\tup{A}})) \subseteq \st\vL{h(\tup{A})}$ and 
		$f(g(\St \vL{\tup{B}})) \subseteq \st\vL{h(\tup{B})}$, or equivalently (since $f,g$ are $\ZZ_2$-maps), $f(g(\St \vR{\tup{B}})) \subseteq \st\vR{h(\tup{B})}$.
	Since $\vL{\tup{A}}$ is contained in both $\St\vL{\tup{A}}$ and $\St\vR{\tup{B}}$,
	$f(g(\vL{\tup{A}}))$ is contained in both $\st\vL{h(\tup{A})}$ and $\st\vR{h(\tup{B})}$.
	Hence $\st\vL{h(\tup{A})} \cap \st\vR{h(\tup{B})} \neq \emptyset$, which implies that $h(\tup{A})$ and $h(\tup{B})$ must be adjacent graph vertices.

	Finally we show that the $\ZZ_2$-map induced by $h$ is $\ZZ_2$-homotopic to $f \circ g$, which in turn is $\ZZ_2$-homotopic to $f \circ |p_k|$, as guaranteed by Theorem~\ref{thm:technical}.
	Indeed, let $x$ be a point in a simplex $|\sigma| = |\{\vL{\tup{A}^{1}},\dots,\vL{\tup{A}^{n}},\vR{\tup{B}^{1}},\dots,\vR{\tup{B}^{m}}\}|$ of $|\Bx{\PathRF{2k+1}(G)}|$.
	Then $h(x)$ is a point in the simplex $|h(\sigma)|$, meaning $|\{\vL{h(\tup{A}^1)}, \dots,  \vR{h(\tup{B}^1)}, \dots\}|$ (since the $\ZZ_2$-map induced by the graph homomorphism is defined as a linear extension of the map on vertices).
	On the other hand $x$ is in $\St \vL{\tup{A}^{1}} \cap \dots \cap \St  \vR{\tup{B}^{1}} \cap \dots$, hence $f(g(x))$ is in $\st \vL{h(\tup{A}^1)} \cap \dots \st \vR{h(\tup{B}^1)} \cap \dots$, which is equal to $\st \{\vL{h(\tup{A}^1)}, \dots,  \vR{h(\tup{B}^1)}, \dots\} = \st h(\sigma)$ (see Lemma~2C.2. in~\cite{Hatcher}).
	Therefore $h(x)$ and $f(g(x))$ are both contained in a common simplex (a simplex in the star of $h(\sigma)$).
	We can thus define a $\ZZ_2$-homotopy $t \cdot f(g(x)) + (1-t) \cdot h(x)$ (this is clearly continuous for $x$ varying on any simplex $|\sigma|$, hence everywhere) from $h$ to $f \circ g$.
\end{proof}